\documentclass[a4paper]{article}

\newcommand{\mystyle}{numeric}

\renewcommand{\mystyle}{numeric}

\usepackage[style=\mystyle,natbib=true]{biblatex}
\usepackage{longtable,booktabs,array}
\usepackage{amssymb,amsmath,amsthm}
\usepackage{stmaryrd}
\usepackage[utf8]{inputenc}
\usepackage{utf8}
\usepackage[all]{xy}

\usepackage{tikz-cd}
\usepackage{graphicx}
\usepackage[symbol]{footmisc}
\usepackage{color}

\usepackage{notation}
\newtheorem{lemma}{Lemma}
\newtheorem{theorem}{Theorem}
\newtheorem{example}{Example}
\newtheorem{definition}{Definition}
\newtheorem{proposition}{Proposition}
\newtheorem{corollary}{Corollary}
\newtheorem{assumption}{Assumption}

\usepackage{agdahen}

\usepackage{hyperref}
\hypersetup{
            pdftitle={Univalent material set theory},
            pdfkeywords={type theory, HoTT, constructive set
theory, material set theory},
            pdfborder={0 0 0},
            breaklinks=true}

\usepackage[margin=25mm]{geometry}

\providecommand{\tightlist}{%
  \setlength{\itemsep}{0pt}\setlength{\parskip}{0pt}}

\title{Univalent Material Set Theory}

\author{Håkon Robbestad Gylterud \and Elisabeth Stenholm}

\begin{document}

\maketitle

\scrollmode

\begin{abstract}

  Homotopy type theory (HoTT) can be seen as a generalisation of structural set
  theory, in the sense that 0-types represent structural sets within the more
  general notion of types. For material set theory, we also have concrete models
  as 0-types in HoTT, but this does not currently have any generalisation to
  higher types. The aim of this paper is to give such a generalisation of material
  set theory to higher type levels within homotopy type theory. This is achieved
  by generalising the construction of the type of iterative sets
  \cite{gylterud-iterative} to obtain an $n$-type universe of $n$-types. At level
  1, this gives a connection between groupoids and multisets.
  
  More specifically, we define the notion of an ∈-structure as a type with an
  extensional binary type family and generalise the axioms of constructive set
  theory to higher type levels. There is a tight connection between the univalence
  axiom and the extensionality axiom of ∈-structures. Once an ∈-structure is
  given, its elements can be seen as representing types in the ambient type
  theory. A useful property of these structures is that an ∈-structure of n-types
  is itself an n-type, as opposed to univalent universes, which have higher type
  levels than the types in the universe.
  
  The theory has an alternative, coalgebraic formulation, in terms of coalgebras
  for a certain hierarchy of functors, $Pⁿ$, which generalises the powerset
  functor from sub-types to covering spaces and $n$-connected maps in general. The
  coalgebras which furthermore are fixed-points of their respective functors in
  the hierarchy are shown to model the axioms given in the first part.
  
  As concrete examples of models for the theory developed we construct the initial
  algebras of the $Pⁿ$ functors. In addition to being an example of initial
  algebras of non-polynomial functors, this construction allows one to start with
  a univalent universe and get a hierarchy of ∈-structures which gives a
  stratified ∈-structure representation of that universe. These types are moreover
  $n$-type universes of $n$-types which contain all the usual types an type
  formers. The universes are cumulative both with respect to universe levels and
  with respect to type levels.
  
  All the results are formalised in the proof-assistant Agda.
  
\end{abstract}

\hypertarget{introduction}{%
\section{\texorpdfstring{Introduction
\label{introduction}}{Introduction }}\label{introduction}}

\emph{Material set theories} are set theories which emphasise the notion
of sets as collections of elements (often themselves sets) and where the
identity of individual elements is tracked across sets, usually with a
global, binary membership relation (∈). This category includes the
traditional Zermelo--Fraenkel (ZF) set theory. ZF is a theory in the
language of first-order logic, intended as a foundation for mathematics.
In what follows, we work within the framework of homotopy type theory
(HoTT), fully formalized in Agda. HoTT is a structural framework, with
the Univalence Axiom in particular allowing identification of types
which are structurally the same, i.e. equivalent -- deemphasising the
individual elements and their identity outside the structure of the
type. Taking a step back, one can see HoTT as a generalisation of
structural set theory to higher type levels. The high-level question we
attempt to answer in this paper is: What is the corresponding
generalisation of material set theory to higher type levels? From this
vantage point we will regard material set theory in the same way that a
classical mathematician regards group theory. Namely, we study a certain
type of mathematical structures, and interest ourselves in their
properties and relationships. The structures we study are not groups,
but what we call ∈-structures: types with an extensional elementhood
relation.

Since we are working in HoTT, we can consider the type level of the
underlying type of sets and of the relation of an ∈-structure. In
classical set theory, the statement ``\(x ∈ y\)'' is a proposition. But
in our setting, we can consider ∈-structures where \(x ∈ y\) is a type
of any type level. Similarly, we can generalise the axioms of ZF to
higher type levels. An example of such a ∈-structure, where \(x ∈ y\) is
allowed to be a type of any level, was considered in \emph{Multisets in
type theory} \cite{gylterud-multisets} by one of the authors.

We aim to give a higher level generalisation of material set theory, by
considering ∈-structures where \(x ∈ y\) is an \(n\)-type. With care,
the usual properties, which we know and love from set-theory, can be
reformulated and proven to hold in our models. But sometimes what used
to be a single property generalises to several properties when taking
higher type levels into consideration. Let us, for the sake of building
some intuition, take a closer look at level 1 of this generalisation.

In a level 1 ∈-structure, elements are related by a set-valued
∈-relation: given two elements, \(x,y : V\) the type \(x ∈ y\) is a set.
One way of understanding this to think of \(y\) as a multiset where
\(x\) may occur more than once. For instance, if \(x ∈ y\) is a finite
type with \(n\) elements, then we can think of this as saying that \(x\)
occurs \(n\) times in \(y\). The generalised properties support this
interpretation: For instance, level 1 (unordered) tupling allows the
formation of multisets of the form \(\{x₀,⋯,x_n\}₁\) where repetitions
will be counted separately. But a level 1 ∈-structure may also support
level 0 (unordered) tupling, with a separate operation,
\(\{x₀,⋯,x_n\}₀\), which becomes a set: \(x∈\{x₀,⋯,x_n\}₀\) being
proposition for any \(x\).

The connection between level 1 ∈-structures and level 1 types,
i.e.~groupoids, is the (perhaps surprising) observation that these
multisets represent groupoids. First of all, a level 1 ∈-structure is
itself a groupoid: The identifications between multisets are free to
permute the occurrences of a given element, giving rise to non-trivial
automorphisms. For instance, the type \(\{∅,∅\}₁ =\{∅,∅\}₁\) has two
distinct elements. A consequence of this is that if we look at the total
type of elements of a multiset, \(\El␣x := ∑_{y:V}y∈x\), we get a
groupoid -- the groupoid represented by \(x\). At first glance, it might
seem as though \(\El␣x\) might always be a set. For instance,
\(\El␣\{∅,∅\}₁\) is a set with two elements. But, by nesting multisets,
we can represent other groupoids. For instance, the cyclic group with
two elements (as a groupoid) is represented by \(\{\{∅,∅\}₁\}₀\), the
set which contains the multiset \(\{∅,∅\}₁\) exactly once. The reason
why \(\El␣\{\{∅,∅\}₁\}₀ = B(ℤ₂)\) is a bit subtle. Notice, the
alternation of subscripts on the tuplings. Had we instead chosen
\(\{\{∅,∅\}₁\}₁\), we would have two occurrences of \(\{∅,∅\}₁\),
because of its two automorphisms, while (perhaps counter-intuitively)
\(\El␣\{\{∅,∅\}₁\}₁\) is the unit type. When we do a 0-singleton of a
multiset, say \(x\), the total type is in general the connected
component, because \(y ∈ \{x\}_0 ≃ ∥y=x∥_{-1}\) and hence
\(\El␣\{x\}₀ ≃ ∑_{y:V}∥y=x∥_{-1}\). So, if \(x\), as \(\{∅,∅\}₁\) does,
has non-trivial automorphisms, these will show up in \(\El␣\{x\}₀\). In
a strong enough level 1 ∈-structure, any (small) group can be
represented.

There is an immediate connection between univalent set theory and
homotopy type theory, whereby there is an equivalence between
∈-structures and coalgebras for the \(n\)-truncated maps functors
\(\T^{n+1}_U : \Type → \Type\), which associates to each type \(X\) the
type of n-truncated maps into \(X\). Thus, \(\T⁰_U X\) is the type of
subtypes of \(X\), while \(\T¹_UX\) is the type of covering spaces of
\(X\) and so on. We show that these functors have initial algebras,
\(\Vⁿ\), and determine the univalent set theory axioms satisfied by the
initial algebras and other fixed-points of these functors. These initial
algebras hence form a family of models of univalent material set theory,
motivating the axioms and interpolating between the standard iterative
hierarchy and the generalised multisets.

One way in which univalent material set theory distinguishes itself in
HoTT is that type levels are \emph{off by one}. What this means is that
models form \(n\)-type based families of \(n\)-types: if \(A,B : V\) are
\(n\)-types in \((V,∈)\) (a notion made precise in Definition
\ref{element-n-type}) then \(A=B\) is an \(n-1\)-type. This means that
\(V\) is an \(n\)-type. This contrasts the situation for univalent
universes, where a well-known result \cite{kraus_truncation_2015} states
that if \(U\) contains strict \(n\)-types the type level of \(U\) is
itself at least \(n+1\).

The model given by the initial algebra \(\Vⁿ\) is of level \(n\), and
thus the type \(\Vⁿ\) is an \(n\)-type. The type can be equipped with
the structure of a Tarski style universe. The decoding of an element in
\(\Vⁿ\) is an \(n\)-type, so \(\Vⁿ\) is an \(n\)-type universe of
\(n\)-types. Moreover, the decoding holds up to definitional equality,
making it very ergonomic to use.

Type levels being off by one might seem strange at first, but it is
caused by the \(∈\)-relation imposing extra structure and thereby
killing automorphisms. This observation generalises what is known about
the cumulative hierarchy in models of (usual) set theory in HoTT, where
\(V\) is a set of sets. Especially in category theory, this can be
useful to strictify structures. For instance, as explored in \emph{The
Category of Iterative Sets in Homotopy Type Theory and Univalent
Foundations} \cite{gratzer2024category}, when recreating the category
with family structure on sets in HoTT, one finds oneself blocked by the
fact that the types in a context forms a strict groupoid, not a set. By
using an ∈-structure as the category of contexts, the off-by-one
property sidesteps this block, yielding a good category with families.

Another perspective we explore is extracting types from ∈-structures. A
given element of an ∈-structure has a type of elements, and considering
the whole ∈-structure we can ask what types can be represented as types
of elements within it. Some insights into how replacement affects
representations of types, such as ℕ for the axiom of infinity, has been
collected in Section \ref{representations-of-types-in-e-structures}. In
particular, we show that the replacement property in set theory says
that the \(∈\)-structure supports all choices of representations of a
type equally (Proposition \ref{internalisation-from-replacement}).

\hypertarget{contributions}{%
\subsection{Contributions}\label{contributions}}

The following are the main contributions of the paper.

\begin{itemize}
\tightlist
\item
  Construction of initial algebras for the non-polynomial functors
  \(\Tⁿ_U\), generalising the construction of the type of iterative of
  sets as the initial algebra for the powerset functor to higher type
  levels (Theorem \ref{V-n-initial}).
\item
  Proof that these initial algebras are \(n\)-type universes of
  \(n\)-types, with definitional decoding (Section
  \ref{section-tarski-universe}).
\item
  Generalisation of the axioms of set theory to properties of
  ∈-structures of any type level (Section \ref{section-e-structures}).
\item
  A framework for representations of types in ∈-structures. This is
  applied to give a new formulation of the axiom of infinity, which does
  not fix a specific encoding of the natural numbers (Definition
  \ref{natural-numbers}).
\item
  Equivalence of \(\Tⁿ_U\)-coalgebras and \(U\)-like ∈-structures,
  generalising the well-known connection between coalgebra and set
  theory (Theorem \ref{U-like-equiv-T-n-coalgebra}).
\item
  Proof that any fixed-point of \(\Tⁿ_U\) is a model of the
  generalisations of the axioms of set theory, except foundation, both
  generalising and proving in HoTT the result by Rieger
  \cite{rieger1957} (Section \ref{section-fixed-point-models}).
\item
  New and short proof of the fiberwise equivalence lemmas: equivalence
  of families of maps (resp. equivalences) and maps (resp. equivalences)
  of total spaces respecting the first coordinate (Lemma
  \ref{total-fiberwise-map} and Corollary \ref{total-fiberwise-equiv}).
\end{itemize}

Some of the ideas and definitions of this article were present in an
unpublished preprint, titled ``Non-wellfounded sets in HoTT''
\cite{GylterudBonnevier2020}. This preprint however, had a flawed
argument in its fourth section and the main construction of that
preprint cannot be carried out as described there. The results from
Sections 2 and 3 of the preprint, which were correct, have been
generalised to higher type levels. These generalisations can now be
found in Section 2 and 5 of the current paper.

\hypertarget{formalisation}{%
\subsection{Formalisation}\label{formalisation}}

Everything in this paper has been formalised in the \texttt{Agda} proof
assistant \cite{agda}. Our formalisation builds on the
\texttt{agda-unimath} library \cite{agda-unimath}, which is an extensive
library of formalised mathematics from the univalent point of view.

The formalisation for this paper can be found at:
\url{https://git.app.uib.no/hott/hott-set-theory}. Throughout the paper
there will also be clickable links to specific lines of Agda code
corresponding to a given result. These will be shown as the Agda logo
\AgdaHen.

\hypertarget{notation-and-universes}{%
\subsection{Notation and universes}\label{notation-and-universes}}

A lot of the basic constructions within HoTT have an established
notation at this point in time. Nevertheless, to avoid confusion, we
include here a list of some of the, perhaps less obvious, notation we
will use in this paper. The notation we do not include in this list will
usually follow the conventions of the HoTT Book \cite{hottbook}.

\begin{itemize}
\tightlist
\item
  \(\emptytype\) denotes the empty type, with eliminator \(\exfalso\).
\item
  \(\unittype\) denotes the unit type.
\item
  \(\twoelemtype\) denotes the type with two elements.
\item
  \(\Nat\) denotes the type of natural numbers, with constructors \(0\)
  and \(\sucN\).
\item
  \(\idequiv\) denotes the identity equivalence, on a given type.
\item
  \(f ∼ g\) denotes the type of homotopies from \(f\) to \(g\):
  \(∏_{x : X} f␣x = g␣x\).
\item
  \(\reflhtpy\) denotes the homotopy \(f ∼ f\) given by the map
  \(λ␣x.␣\refl\).
\item
  \(\fib␣f␣y\) denotes the homotopy fiber: \(∑_{x : X} f␣x = y\).
\item
  \(π₀\) and \(π₁\) denote the first, respectively second, projection
  out of a Σ-type.
\item
  Given a path \(p : x = y\), \(p^{-1}\) denotes the inverse path
  \(y = x\).
\item
  Given a family \(P\) of types over \(X\) and a path \(p : x = y\),
  \(\tr{P}{p}{}\) denotes the transport function from \(P␣x\) to \(P␣y\)
  over \(p\).
\item
  Given types \(A\) and \(B\) and a path \(p : A = B\),
  \(\coe␣p : A → B\) is the map defined by path induction, taking the
  identity map for \(\refl\).
\item
  Given an invertible function \(f\) (usually an equivalence),
  \(f^{-1}\) denotes the inverse.
\item
  Given a family of maps \(f : ∏_{x : X} P␣x → Q␣x\),
  \(\total␣f : ∑_{x : X} P␣x → ∑_{x : X} Q␣x\) is the function:
  \(λ␣(x,p).␣(x,f␣x␣p)\).
\item
  \(\isntruncmap{n} f\) is the proposition that \(f\) is an
  \(n\)-truncated map: \(∏_{y : Y} \isntype{n} (\fib␣f␣y)\).
\item
  \(X ↪ Y\) is the type of propositionally truncated maps:
  \(∑_{f : X → Y} \isntruncmap{(-1)} f\).
\item
  \(X ↪_n Y\) is the type of \(n\)-truncated maps:
  \(∑_{f : X → Y} \isntruncmap{n} f\).
\item
  \(X ↠ Y\) is the type of \((-1)\)-connected maps:
  \(∑_{f : X → Y} ∏_{y : Y} \iscontr \| \fib␣f␣y \|_{-1}\).
\item
  \(X ↠_n Y\) is the type of \(n\)-connected maps:
  \(∑_{f : X → Y} ∏_{y : Y} \iscontr \| \fib␣f␣y \|_n\)
\item
  \(\exuniqueinl{x : X} P␣x\) denotes the type:
  \(\iscontr \left(∑_{x : X} P␣x\right)\).
\item
  \(\funext\) is the function \(f ∼ g → f = g\) given by function
  extensionality.
\item
  \(\ua\) is the function \(X ≃ Y → X = Y\) given by univalence.
\item
  \(\Prop_U\) is the type of all propositions in the universe \(U\),
  i.e.~the type \(∑_{X : U} \isprop X\).
\item
  \(\Set_U\) is the type of all sets in the universe \(U\), i.e.~the
  type \(∑_{X : U} \isset X\).
\item
  More generally, \(\nType{n}_U\) is the type of all \(n\)-types in the
  universe \(U\), i.e.~the type \(∑_{X : U} \isntype{n} X\).
\end{itemize}

We will use the same terminology as the HoTT Book regarding type levels.
But we will also define a notion of \emph{level} for ∈-structures and
elements in (the carrier of) an ∈-structure. This overloading of
terminology should be fine however, since it should be clear from the
context what kind of level we are referring to. The notions of
\emph{mere proposition} and \emph{mere set} are used to denote types of
level \(-1\) and \(0\) respectively, when there is need for clarity.

In this paper we will assume two type universes\footnote{The
  formalisation differs in this respect as it uses a hierarchy of
  universes, but the relationship between \(U\) and \(\Type\) in the
  article, is the same as the relationship between \texttt{UU\ i} and
  \texttt{UU\ (lsuc\ i)} in the formalisation.}, a large univalent
universe, denoted \(\Type\), and a small univalent universe, denoted
\(U\). We use cumulative universes, i.e.~\(U : \Type\) and
\(X : \Type\), for all \(X : U\). It is assumed that both \(U\) and
\(\Type\) are closed under the usual type formers, such as Π-types,
Σ-types, and identity types. The constructions below also use the empty
type and the type of natural numbers. We will use function
extensionality freely.

From Section
\ref{equivalence-of-extensional-coalgebras-and-u-like-e-structures} and
onwards, we will also assume that we can construct small images in
certain situations. This assumption is informed by Rijke's modified join
construction \cite{rijke2017}, which can be used to construct such small
images. One can alternatively assume that \(U\) is closed under homotopy
colimits, from which the smallness of (certain) images follows by the
join construction.

\hypertarget{structures}{%
\section{\texorpdfstring{∈-structures\label{section-e-structures}}{∈-structures}}\label{structures}}

In this section we give the definition of ∈-structures\footnote{This
  section defines the notion of ∈-structure slightly differently than
  ∈-structures were defined in a previous article by one of the authors
  \cite{gylterud-iterative}. This difference is by incorporating
  extensionality and generalising to higher type levels.} and formulate
properties of these inspired by set theory. Most of the properties are
indexed by a type level, from 0 to ∞. The level 0 version of the
property is equivalent to the usual set theoretic concept for
∈-structures of level 0, while the ∞ version was explored in
\cite{gylterud-multisets}.

\begin{definition}[\agdalink{https://elisabeth.stenholm.one/univalent-material-set-theory/e-structure.core.html\#1878}]

An \textbf{∈-structure} is a pair \((V,∈)\) where \(V : \Type\) and
\(∈␣: V → V → \Type\), which is \textbf{extensional}: for each
\(x,y : V\), the canonical map \(x = y → ∏_{z : V} {z ∈ x} ≃ {z ∈ y}\)
is an equivalence of types.\end{definition}

Extensionality states that we can distinguish sets by their elements. It
is expressed in first order logic using logical equivalence, but since
we are working in the framework of HoTT and want to allow for
elementhood relations of higher type level, we use instead equivalence
of types. This of course reduces to logical equivalence in the case when
the ∈-relation is propositional.

Many times we will want to talk about all members of a given element in
\(V\). We introduce a notation for this.

\begin{definition}[\agdalink{https://elisabeth.stenholm.one/univalent-material-set-theory/e-structure.core.html\#5019}]

\label{def-el}Given an ∈-structure, \((V,∈)\), we define the family
\(\El : V → \Type\) by \(\El␣a ≔ ∑_{x : V} x ∈ a\).\end{definition}

The usual notion of (extensional) model of set theory corresponds to
∈-structures, \((V,∈)\), where \(x ∈ y\) is a mere proposition for each
\(x,y : V\) (and consequently \(V\) is a mere set). However, there are
examples of extensional ∈-structures where \(V\) is not a mere set. One
such example, based on Aczel's \cite{aczel1978} type \(\W{A:U}{A}\), was
explored in an article by one of the authors \cite{gylterud-multisets}.

We can stratify ∈-structures based on the type level of the ∈-relation.

\begin{definition}[\agdalink{https://elisabeth.stenholm.one/univalent-material-set-theory/e-structure.core.html\#4010}]

Given \(n : \Nat_{-2}\), an ∈-structure \((V,∈)\) is said to be of
\textbf{level (n+1)} if for every \(x,y:V\) the type \(x ∈ y\) is an
\(n\)-type.\end{definition}

The following proposition explains the occurrence of \(n+1\) in the
definition above:

\begin{proposition}[\agdalink{https://elisabeth.stenholm.one/univalent-material-set-theory/e-structure.core.html\#4100}]

\label{m-h-level}In an ∈-structure, \((V,∈)\), of level \(n\) the type
\(V\) is an \(n\)-type.\end{proposition}

\begin{proof}

Let \((V,∈)\) be an ∈-structure of level \((n+1)\), for
\(n : \Nat_{-2}\). By extensionality \(x = y\) is equivalent to
\(∏_{z : V} {z ∈ x} ≃ {z ∈ y}\), which is an \(n\)-type, hence \(V\) is
an \((n+1)\)-type.\end{proof}

\textbf{Remark:} By definition there are no \(-2\) level ∈-structures,
and the \(-1\) level structures all have trivial ∈-structure. Thus, we
shall from here on focus on the ∈-structures of level 0 or higher.

Of special interest will be the elements of an ∈-structure which look
like sets, in the sense that elementhood is a proposition. This is by
definition the case for all elements in ∈-structures of level 0, but
such elements may occur in structures of all levels.

\begin{definition}[\agdalink{https://elisabeth.stenholm.one/univalent-material-set-theory/e-structure.core.html\#2845}]

\label{element-n-type}An element \(x:V\) is a \((k+1)\)-type in
\((V,∈)\) if \(y∈x\) is of level \(k : \Nat_{-1}\) for all
\(y:V\).\end{definition}

If \(x : V\) is a \(0\)-type in \((V,∈)\), we also say it is a
\emph{mere set} in \((V,∈)\).

\begin{proposition}[\agdalink{https://elisabeth.stenholm.one/univalent-material-set-theory/e-structure.core.html\#3705}]

If \(x : V\) is a \((k+1)\)-type in \((V,∈)\) then \(x = y\) is a
\(k\)-type for any \(y:V\).\end{proposition}

\begin{proof}

By extensionality, the type \(x = y\) has the same type level as the
type \(∏_{z : V} z ∈ x ≃ z ∈ y\). The latter is a \(k\)-type as
\(z ∈ x\) is a \(k\)-type for every \(z : V\).\end{proof}

Note that \(x : V\) being a \(k\)-type in \((V,∈)\) does not imply that
\(\El␣x\) is a \(k\)-type. However, if \(V\) is a \(k\)-type then
\(\El␣x\) is a \(k\)-type if \(x\) is a \(k\)-type.

\hypertarget{ordered-pairing}{%
\subsection{Ordered pairing}\label{ordered-pairing}}

The characteristic property of ordered pairs is that two pairs are equal
if and only if the first coordinates are equal and the second
coordinates are equal, i.e.~it is a pair where the order of the elements
matters. In our setting this means that the ordered pair of two elements
\(x,y : V\) should be an element \(〈x,y〉 : V\) such that for any other
ordered pair \(〈x',y'〉\), for some \(x', y' : V\),
\(〈x,y〉 = 〈x',y'〉\) exactly when \(x=x'\) and \(y=y'\). Using the
characterisation of the identity types of cartesian products, this is
equivalently saying that \(〈x,y〉 = 〈x',y'〉\) exactly when
\((x,y) =_{V × V} (x',y')\). In order to allow for higher level
∈-structures the ``exactly when'' should be replaced with type
equivalence. Moreover, we want it to be the canonical one, in the sense
that if \(x≡x'\) and \(y≡y'\) then the equivalence should send \(\refl\)
to \(\refl\).

In HoTT, this is the statement that ordered pairing is an embedding
\(V × V ↪ V\). This neatly encapsulates and generalises the usual
characterisation of equality of ordered pairs in a way that is
completely independent of the level of the ∈-structure. But it does not
uniquely define the encoding of ordered pairs. Indeed, there are several
ways to encode ordered pairs in ordinary set theory. The usual
Kuratowsky pairing will work for 0-level ∈-structures, but not for
higher level structures. Luckily, Norbert Wiener's encoding,
\(〈x,y〉 ≔ \{\{\{x\},∅\},\{\{y\}\}\}\), will work for structures of any
level.

This was originally observed in previous work by one of the authors
\cite{gylterud-multisets}. However, there the equivalence
\(\left(〈x,y〉 = 〈x',y'〉\right) ≃ \left((x,y) =_{V × V} (x',y')\right)\)
was not required to be the canonical one, as we require here. Moreover,
the encoding of ordered pairs used the untruncated variants of
singletons and unordered pairs. But it may be the case that we can only
construct truncated variants in a given ∈-structure. Therefore, we will
here construct ordered pairs, based on the Wiener encoding, but for any
truncation level of singletons and unordered pairs.

Since the results which follow are independent of encoding, we will not
commit to any specific way of forming ordered pairs, but simply assume
ordered pairing as an extra structure.

\begin{definition}[\agdalink{https://elisabeth.stenholm.one/univalent-material-set-theory/e-structure.core.html\#7351}]

Given an ∈-structure, \((V,∈)\), \textbf{an ordered pairing structure}
on \((V,∈)\) is an embedding \(V × V ↪ V\).\end{definition}

While having ordered pairing is a structure, once the pairing structure
is fixed, the notion of \emph{being an ordered pair} is a proposition.
As is clear from the proof below, this fact is just another formulation
of the characterisation of equality of ordered pairs.

\begin{proposition}

\label{ordered-pair-prop}Being an ordered pair is a mere proposition:
for a fixed ∈-structure \((V,∈)\) with ordered pairing structure
\(〈-,-〉\), the type \(∑_{a,b:V}〈a,b〉 = x\) is a proposition, for all
\(x : V\).\end{proposition}

\begin{proof}

Observe that \(∑_{a,b:V}〈a,b〉 = x\) is the fiber of \(〈-,-〉\) over
\(x\), which is a proposition since \(〈-,-〉\) is an
embedding.\end{proof}

\hypertarget{properties-of--structures}{%
\subsection{Properties of
∈-structures}\label{properties-of--structures}}

In this section we explore how further set-theoretic notions, such as
pairing, union, replacement, separation and exponentiation, can be
expressed as mere propositions about ∈-structures. As mentioned in the
introduction, these notions can be generalised in different ways by
using different levels of truncation. When characterising unions, for
instance, it makes a difference whether the truncated existential
quantifier or dependent pair types are used. With the truncated
existential quantifier we only get one copy of each element in the
union, but with the dependent pair type we may get more copies.

A given ∈-structure can satisfy several versions, but a recurring theme
is that \(n\)-level structures will only satisfy the \(k\)-truncated
versions for \(k ≤ n\).

\textbf{Convention:} \emph{In the rest of the paper we will consider the
type of truncation levels to be the type \(\Nat^∞_{-2}\), i.e.~the usual
truncation levels extended by an element \(∞\) for which \(\| P \|_∞\)
is defined as \(P\) and such that \(∞ - 1 = ∞ = ∞ + 1\).}

\emph{The properties of ∈-structures will be parameterised by truncation
level. If the truncation level is omitted, in the notation or reference
to a property, we mean the variant that is labeled with \(0\) (which
usually involves \((-1)\)-truncation in the definition).}

Interestingly, even the untruncated versions of the set theoretic
properties end up being propositions. For instance, to have all ∞-unions
is a mere property of ∈-structures. This is because the properties
characterise the material sets they claim existence of up to equality,
by extensionality.

\begin{proposition}[\agdalink{https://elisabeth.stenholm.one/univalent-material-set-theory/e-structure.core.html\#2298}]

\label{is-prop-fixed-e-rel}Given an ∈-structure, \((V,∈)\), let
\(φ : V → \Type\) be a type family on \(V\). Then the type
\(∑_{x : V} ∏_{z : V} z ∈ x ≃ φ␣z\) is a proposition.\end{proposition}

\begin{proof}

Assume \((x,α) : ∑_{x : V} ∏_{z : V} z ∈ x ≃ φ␣z\), then it is enough to
show that the type is contractible. We have the following chain of
equivalences: \begin{align*}
    \left(∑_{x' : V} ∏_{z : V} z ∈ x' ≃ φ␣z\right)
        ≃ \left(∑_{x' : V} ∏_{z : V} z ∈ x' ≃ z ∈ x\right)
        ≃ \left(∑_{x' : V} x' = x\right)
\end{align*} The last type is contractible.\end{proof}

\textbf{Remark:} Proposition \ref{is-prop-fixed-e-rel} states that the
generalisation of the unrestricted comprehension, \(\{z | φ␣z\}\),
determines a set uniquely, when existent. Many existence statements in
set theory can be seen as fleshing out for which forms \(φ\) this
comprehension defines a set.

\hypertarget{unordered-tuples}{%
\subsubsection{Unordered tuples}\label{unordered-tuples}}

The usual notion of pairing naturally extends to unordered tupling of
any arity. The arity can be expressed by any type. The usual pairing
operation is tupling with respect to the booleans, the singleton
operation is tupling for the unit type, and the empty set is tupling for
the empty type. We generalise the property of having tupling to all
truncation levels.

\begin{definition}[\agdalink{https://elisabeth.stenholm.one/univalent-material-set-theory/e-structure.property.unordered-tupling.html\#2110}]

Given \(k : \Nat^∞_{-1}\) and a type \(I\), an ∈-structure, \((V,∈)\),
has \textbf{\((k+1)\)-unordered \(I\)-tupling} if for every
\(v : I → V\) there is \(\{v\}_{k+1} : V\) such that
\(∏_{z:V} z∈\{v\}_{k+1} ≃ \| ∑_{i:I} v␣i = z \|_k\)

For the special cases \(I = \Fin␣n\), we say that \((V,∈)\) has
\(k\)-unordered \(n\)-tupling. If \((V,∈)\) has \(k\)-unordered
\(n\)-tupling for every \(n:\Nat\), we say that \((V,∈)\) has finite,
\(k\)-unordered tupling.\end{definition}

We will use the usual notation for finite tuplings, but with a subscript
for the truncation level: \(k\)-unordered \(n\)-tupling is denoted by
\(\{x₀,⋯,x_{n-1}\}_k\). For \(k\)-unordered \(\emptytype\)-tupling we
will use the notation \(∅\). Observe that \(\{x₀,⋯,x_{n-1}\}_k\) is a
\(k\)-type in \((V,∈)\).

The set \(∅\) is the set with no elements. For any \(z : V\) we have:
\begin{align}
    z ∈ ∅ 
        ≃ \Big\| ∑_{i:\emptytype} \exfalso␣i = z \ \Big\|_{k-1}
        ≃ \|\ \emptytype \ \|_{k-1}
        ≃ \emptytype
\end{align} Note here that the truncation level \(k\) does not matter
(hence why we exclude it from the notation \(∅\)). By extensionality,
the sets corresponding to \(k\)-unordered and \(k'\)-unordered
\(\emptytype\)-tupling, for any two \(k\) and \(k'\), are equal. Note
also that \(∅\) is a mere set in any ∈-structure.

\textbf{Remark:} We say that an ∈-structure, \((V,∈)\), has
\textbf{empty set} if there is an element \(x : V\) such that
\(∏_{z : V} z ∈ x ≃ \emptytype\). By the previous paragraph, this is
equivalent to saying that \((V,∈)\) has \(\emptytype\)-tupling.

However, the singletons \(\{x\}_k\) may be different for different
truncation levels. For any \(x, z : V\) we have: \begin{align}
    z ∈ \{x\}_k
        ≃ \Big\| ∑_{i:\unittype} x = z \ \Big\|_{k-1}
        ≃ \|\ x = z\ \|_{k-1}
\end{align} In the special case when \(x\) is a mere set in \((V,∈)\),
the truncation level does not matter in that \(\{x\}_k = \{x\}_0\), by
extensionality. As an example: \(\{∅\}_k = \{∅\}_{k'}\) for any two
\(k\) and \(k'\), since \(∅\) is a mere set. Starting at level 1,
repetitions matter in tupling: Given \(k > 0\) we have
\(\{∅,∅\}_k ≠ \{∅\}_k\), since
\(∅ ∈ \{∅,∅\}_k ≃ ∥ {∅ = ∅} + {∅ = ∅}∥_k ≃ 2\) while
\(∅ ∈ \{∅\}_k ≃ ∥∅ = ∅∥_k ≃ 1\).

An interesting usage of 0-unordered 1-tupling combined with 1-unordered
\(n\)-tupling is the construction of a set \(s_n\) such that
\(\El␣s_n = B(S_n)\), the classifying type of the symmetric group on
\(n\) elements. Simply let \(s_n = \{\{n∅\}₁\}₀\), where \(n∅ : Vⁿ\) is
the vector with \(n\) copies of ∅.

\hypertarget{ordered-pairs-from-unordered-tuples}{%
\subsubsection{Ordered pairs from unordered
tuples}\label{ordered-pairs-from-unordered-tuples}}

As noted above, ordered pairs can be constructed from the empty set,
singletons and unordered pairs, using Norbert Wiener's encoding:
\(〈x,y〉 ≔ \{\{\{x\},∅\},\{\{y\}\}\}\). In ∈-structures of arbitrary
level it makes sense to ask what level of truncation \(\{-\}_k\) and
\(\{-,-\}_{k'}\) should be used when defining ordered pairs. The
construction using the \((-1)\)-truncated variants does not work in
higher level structures since
\(\{\{\{x\}₀,∅\}₀,\{\{y\}₀\}₀\}₀ = \{\{\{x'\}₀,∅\}₀,\{\{y'\}₀\}₀\}₀\) is
a proposition, while \((x,y) = (x',y')\) need not be.

If \((V,∈)\) has level \(n\), then we know that \((x,y) = (x',y')\) has
type level \(n-1\). For that case the \((n-1)\)-truncated variants would
give us the correct type level for \(〈x,y〉 = 〈x',y'〉\). But, we
observe that \(\{-\}_n = \{-\}_∞\) in ∈-structures of level \(n\) since
we have \(z ∈ \{x\}_n ≃ \|z = x\|_{n-1} ≃ (z = x)\). For unordered pairs
we need to distinguish between the \((-1)\)-truncated case and all other
cases since coproducts are not closed under propositions. More
specifically, for \(n ≥ 1\) we have \(\{-,-\}_{n} = \{-,-\}_∞\) if
\((V,∈)\) has level \(n\), since
\(z ∈ \{x,y\}_{n} ≃ \|(z = x) + (z = y)\|_{n-1} ≃ ((z = x) + (z = y))\).
This equivalence does not hold for \(n = 0\) and arbitrary \(x,y : V\).
However, if \(x ≠ y\), then \(z = x → z ≠ y\) and the equivalence holds.
We use these observations to make a general construction of ordered
pairs which we can instantiate for ∈-structures of all levels.

\begin{lemma}[\agdalink{https://elisabeth.stenholm.one/univalent-material-set-theory/e-structure.property.unordered-tupling.html\#16024}]
\label{ord-pair-from-embs-lma}

Given an ∈-structure, \((V,∈)\), with an operation \(α : V × V → V\)
with equivalences
\(e : ∏_{x, y : V} x ≠ y → ∏_{z : V} z ∈ α(x, y) ≃ ((z = x) + (z = y))\),
and two disjoint embeddings, \(f\ g : V ↪ V\),
i.e.~\(∏_{x, y : V} f␣x ≠ g␣y\), then \(α ∘ (f × g): V × V → V\) is an
embedding.\end{lemma}

\begin{proof}

We need to show, for any \((x, y), (x', y') : V × V\), that
\(\ap{α ∘ (f × g)} : (x, y) = (x' , y') → α(f␣x, g␣y) = α(f␣x', g␣y')\)
is an equivalence. To this end, it is enough to construct some
equivalence between the identity types that sends
\(\refl : α(f␣x, g␣y) = α(f␣x, g␣y)\) to \(\refl : (x, y) = (x , y)\).

We have the following chain of equivalences:

\begin{align}
\left(α(f␣x, g␣y) = α(f␣x', g␣y')\right)
    \label{ord-pair-equiv-1}
    &≃ ∏_{z : V} z ∈ α(f␣x, g␣y) ≃ z ∈ α(f␣x', g␣y') \\
    \label{ord-pair-equiv-2}
    &≃ ∏_{z : V} ((z = f␣x) + (z = g␣y)) ≃ ((z = f␣x') + (z = g␣y')) \\
    \label{ord-pair-equiv-3}
    &≃ ∏_{z : V} ((z = f␣x) ≃ (z = f␣x')) × ((z = g␣y) ≃ (z = g␣y'))\\
    &≃ \left(∏_{z : V} (z = f␣x) ≃ (z = f␣x')\right)
     × \left(∏_{z : V} (z = g␣y) ≃ (z = g␣y')\right) \\
    &≃ \left(f␣x = f␣x'\right) × \left(g␣y = g␣y'\right) \\
    \label{ord-pair-equiv-4}
    &≃ \left(x = x'\right) × \left(y = y'\right) \\
    &≃ ((x, y) = (x', y'))
\end{align}

In step (\ref{ord-pair-equiv-1}) we use extensionality for \((V,∈)\). In
step (\ref{ord-pair-equiv-2}) we use the equivalences \(e␣(f␣x)␣(g␣y)\)
and \(e␣(f␣x')␣(g␣y')\), together with the fact that \(f␣x ≠ g␣y\) and
\(f␣x' ≠ g␣y'\). The equivalence (\ref{ord-pair-equiv-3}) follows from
the fact that \(z = f␣x\) and \(z = g␣y'\), and \(z = f␣x'\) and
\(z = g␣y\), are, respectively, mutually exclusive. In step
(\ref{ord-pair-equiv-4}) we use the fact that \(f\) and \(g\) are
embeddings.

We chase \(\refl : α(f␣x, g␣y) = α(f␣x, g␣y)\) through the equivalence:

\begin{align}
\refl &↦ λ␣z. \idequiv \\
      &↦ λ␣z.\left(e␣z\right) ∘ \idequiv ∘ \left(e␣z\right)^{-1} \\
      &=\ λ␣z. \idequiv \\
      &↦ λ␣z.(\idequiv, \idequiv) \\
      &↦ (λ␣z.\idequiv, λ␣z.\idequiv) \\
      &↦ (\refl, \refl) \\
      &=\ \left(\ap{f}␣\refl, \ap{g}␣\refl\right) \\
      &↦ \left(\ap{f}^{-1}\left(\ap{f}␣\refl\right), \ap{g}^{-1}\left(\ap{g}␣\refl\right)\right) \\
      &=\ (\refl, \refl) \\
      &↦\ \refl
\end{align}

where we have used the fact that extensionality for \((V,∈)\) sends
\(\refl\) to \(\idequiv\).\end{proof}

For the Norbert Wiener construction of ordered pairs we thus have to
show that both \(\{\{-\}_{n}\}_{n}\) and \(\{\{-\}_{n},∅\}_{n}\) are
embeddings. Using the previous observation about the relationship
between the \(n\)-truncated versions and the ∞-truncated ones, it is
enough to show this for the ∞-truncated versions.

\begin{lemma}[\agdalink{https://elisabeth.stenholm.one/univalent-material-set-theory/e-structure.property.unordered-tupling.html\#9770}]
\label{singleton-emb-lma}

The function \(\{-\}_∞ : V → V\) is an embedding.\end{lemma}

\begin{proof}

We follow the same strategy as in the proof of Lemma
\ref{ord-pair-from-embs-lma}. For any \(x, y : V\) we have the following
chain of equivalences:

\begin{align}
\left(\{x\}_∞ = \{y\}_∞\right)
    \label{singleton-emb-equiv-1}
    &≃ ∏_{z : V} z ∈ \{x\}_∞ ≃ z ∈ \{y\}_∞ \\
    &≃ ∏_{z : V} \left(z = x\right) ≃ \left(z = y\right) \\
    &≃ \left(x = y\right)
\end{align}

In step (\ref{singleton-emb-equiv-1}) we use extensionality for
\((V,∈)\).

Let \(e : ∏_{x : V} ∏_{z : V} z ∈ \{x\}_∞ ≃ \left(z = x\right)\) be the
defining family of equivalences for \(\{-\}_∞\). We chase
\(\refl : \{x\}_∞ = \{x\}_∞\) along the chain of equivalences above:

\begin{align}
\refl \label{singleton-emb-map-1}
    &↦ λ␣z. \idequiv \\
    &↦ λ␣z.(e␣x␣z) ∘ \idequiv ∘ (e␣x␣z)^{-1} \\
    &=\ λ␣z.\idequiv \\
    &↦ \refl
\end{align}

In step (\ref{singleton-emb-map-1}) we use the fact that extensionality
for \((V,∈)\) sends \(\refl\) to \(\idequiv\).\end{proof}

\begin{lemma}[\agdalink{https://elisabeth.stenholm.one/univalent-material-set-theory/e-structure.property.unordered-tupling.html\#12350}]
\label{lma-unordered-pair-emb}

Let \(α : V → V → V\) be such that
\(e : ∏_{x, y : V} x ≠ y → ∏_{z : V} z ∈ α␣x␣y ≃ ((z = x) + (z = y))\).
Then \(λ␣x.␣α␣\{x\}_∞␣∅ : V → V\) is an embedding.\end{lemma}

\begin{proof}

First, we observe that for all \(x : V\), \(\{x\}_∞ ≠ ∅\) since
\(x ∈ \{x\}_∞\) is inhabited but \(x ∈ ∅\) is empty. We now follow the
same strategy as in the proof of Lemma \ref{ord-pair-from-embs-lma}. For
any \(x, y : V\) we have the following chain of equivalences:

\begin{align}
\left(α␣\{x\}_∞␣∅ = α␣\{y\}_∞␣∅\right)
    \label{pair-sing-empty-equiv-1}
    &≃ ∏_{z : V} (z ∈ α␣\{x\}_∞␣∅) ≃ (z ∈ α␣\{y\}_∞␣∅) \\
    &≃ ∏_{z : V} \left((z = \{x\}_∞) + (z = ∅)\right) ≃ \left((z = \{y\}_∞) + (z = ∅)\right) \\
    \label{pair-sing-empty-equiv-2}
    &≃ ∏_{z : V} \left((z = \{x\}_∞) ≃ (z = \{y\}_∞)\right) × \left((z = ∅) ≃ (z = ∅)\right) \\
    \label{pair-sing-empty-equiv-3}
    &≃ ∏_{z : V} (z = \{x\}_∞) ≃ (z = \{y\}_∞) \\
    &≃ \left(\{x\}_∞ = \{y\}_∞\right) \\
    \label{pair-sing-empty-equiv-4}
    &≃ \left(x = y\right)
\end{align}

In step (\ref{pair-sing-empty-equiv-1}) we use extensionality for
\((V,∈)\). In step (\ref{pair-sing-empty-equiv-2}) we use the fact that
\(\{x\}_∞ ≠ ∅\) and \(\{y\}_∞ ≠ ∅\). In step
(\ref{pair-sing-empty-equiv-3}) we use the fact that
\((z = ∅) ≃ (z = ∅)\) is contractible since ∅ is a mere set, and hence
\(z = ∅\) is a proposition. In step (\ref{pair-sing-empty-equiv-4}) we
use Lemma \ref{singleton-emb-lma}.

We chase \(\refl : α␣\{x\}_∞␣∅ = α␣\{x\}_∞␣∅\) along the chain of
equivalences above:

\begin{align}
\refl
    \label{pair-sing-empty-map-1}
    &↦ λ␣z.\idequiv \\
    &↦ λ␣z.(e␣z) ∘ \idequiv ∘ (e␣z)^{-1} \\
    &=\ λ␣z.\idequiv \\
    &↦ λ␣z.(\idequiv, \idequiv) \\
    &↦ λ␣z.\idequiv \\
    &↦ \refl \\
    &=\ \ap{\{-\}_∞}␣\refl \\
    &↦ \ap{\{-\}_∞}^{-1}␣\left(\ap{\{-\}_∞}␣\refl\right) \\
    &=\ \refl
\end{align}

In step (\ref{pair-sing-empty-map-1}) we use the fact that
extensionality for \((V,∈)\) sends \(\refl\) to \(\idequiv\).\end{proof}

\begin{theorem}[\agdalink{https://elisabeth.stenholm.one/univalent-material-set-theory/e-structure.property.unordered-tupling.html\#17726}]

\label{0-level-ordered-pairs}If \((V,∈)\) is an ∈-structure of level
\(n\) which has ∅, \(\{-\}_{n}\) and \(\{-,-\}_{n}\), then it has an
ordered pairing structure given by
\(λ␣(x,y).\{\{\{x\}_{n},∅\}_{n},\{\{y\}_{n}\}_{n}\}_{n}\).\end{theorem}

\begin{proof}

Corollary of Lemma \ref{ord-pair-from-embs-lma}, Lemma
\ref{singleton-emb-lma} and Lemma
\ref{lma-unordered-pair-emb}.\end{proof}

\hypertarget{restricted-separation}{%
\subsubsection{Restricted separation}\label{restricted-separation}}

In constructive set theory \cite{aczel1978}, restricted separation is
the ability to construct sets of the form \(\{z ∈ x␣|␣Φ␣z\}\), for a
formula Φ where all quantifiers are bounded (i.e.~\(∀a∈b⋯\) and
\(∃a∈b⋯\)). One way to internalise this to type theory is to require
that the predicate \(Φ : V → \Prop_U\) is a predicate of propositions in
\(U\). We do the same here in that we require the predicate to take
values in \(U\), but as with the other properties, we generalise to
higher truncation levels.

\begin{definition}[\agdalink{https://elisabeth.stenholm.one/univalent-material-set-theory/e-structure.property.restricted-separation.html\#616}]

An ∈-structure, \((V,∈)\), has \textbf{\(U\)-restricted
\((k+1)\)-separation}, for \(k : \Nat^∞_{-1}\), if for every \(x : V\)
and \(P : \El x → \nType{k}_U\) there is an element \(\{x␣|␣P\} : V\)
such that
\(∏_{z:V} z ∈ \{x␣|␣P\} ≃ ∑_{e : z ∈ x} P␣(z,e)\).\end{definition}

\textbf{Remark:} Usually, in set theory, the predicate is defined for
any \(x\), even when taking the restricted separation. But this falls
under the above: Assume \(Q : V → \nType{k}_U\), then one can readily
define \(P(z,e) = Q␣z\) and the defining property becomes
\(z ∈ \{x␣|␣P\} ≃ (z ∈ x × Q␣z)\). However, going in the opposite
direction only works at level 0, because the predicate \(P\) may
otherwise depend on the specific witness of elementhood.

\textbf{Remark:} At level 1, we can use \(0\)-unordered \(1\)-tupling,
1-unordered tupling and \(U\)-restricted \(1\)-separation, to construct
for every group \(G\) a set \(s_G\) such that \(\El␣s_G = B(G)\), the
classifying type for the group.\footnote{For a thorough treatment of
  groups and classifying types, see the \emph{Symmetry} book
  \cite{bezem2022symmetry}} First let \(|G|∅ : |G| → V\) be the constant
function mapping every element of the group to ∅. This is the tuple to
which we will apply 1-unordered tupling, constructing \(\{|G|∅\}₁\).
This multiset has one copy of ∅ for every element of the group, hence
its automorphisms in \(V\) are the bijections on \(|G|\):

\begin{align}
    \left(\{|G|∅\}₁ = \{|G|∅\}₁\right) &≃ ∏_{z:V} (z∈\{|G|∅\}₁) ≃ (z∈\{|G|∅\}₁) \\
                         &≃ ∏_{z:V} (|G|×(z=∅)) ≃ (|G|×(z=∅)) \\
                         &≃ \left(|G|=_U|G|\right)
\end{align}

If we take the 0-singleton \(\{\{|G|∅\}₁\}₀\), we get a set which has
\(B(\operatorname{Aut}_U␣|G|)\) as its type of elements:

\begin{align}
 \El␣\{\{|G|∅\}₁\}₀ &= ∑_{z:V}z∈\{\{|G|∅\}₁\}₀ \\
                    &≃ ∑_{z:V}∥z = \{|G|∅\}₁ ∥_{-1} \\
                    &= B(\operatorname{Aut}_V␣\{|G|∅\}₁)\\
                    &≃ B(\operatorname{Aut}_U␣|G|)
\end{align}

The final ingredient is that we need a map
\(f : B(G) → B(\operatorname{Aut}␣|G|)\). This map is the well-known map
induced by multiplication in the group. This is a cover, i.e. its fibers
are sets, hence \(\fib␣f : B(\operatorname{Aut}␣|G|) → \nType{0}\),
which we will coerce along the equivalence
\(B(\operatorname{Aut}_U␣|G|) ≃ \El␣\{\{|G|∅\}₁\}₀\) to obtain
\(P_G : \El␣\{\{|G|∅\}₁\}₀ → \nType{0}\). Thus, we define
\(s_G := \{\{\{|G|∅\}₁\}₀|P_G\}\), which then has the desired property:

\begin{align}
    \El␣s_G &= ∑_{z:V}z∈\{\{\{|G|∅\}₁\}₀|P_G\} \\
            &≃ ∑_{z:V}∑_{e:z∈\{\{|G|∅\}₁\}₀}  P_G(z,e) \\
            &≃ ∑_{q : \El␣\{\{|G|∅\}₁\}₀} P_G␣q \\
            &≃ ∑_{q : B(\operatorname{Aut}␣|G|)} \fib␣f␣q \\
            &≃ B(G)
   \end{align}

\hypertarget{replacement}{%
\subsubsection{Replacement}\label{replacement}}

Replacement says that given a set \(a\), and a description of how to
replace its elements, we can create a new set containing exactly the
replacements of elements in \(a\). Of course, two elements in \(a\) may
be replaced by the same elements, so the property must include
truncation.

\begin{definition}[\agdalink{https://elisabeth.stenholm.one/univalent-material-set-theory/e-structure.property.replacement.html\#642}]

An ∈-structure, \((V,{∈})\), has \textbf{\((k+1)\)-replacement}, for
\(k : \Nat^∞_{-1}\), if for every \(a:V\) and \(f : \El␣a → V\) there is
an element \(\{f(x)␣|␣x∈a\} : V\) such that
\(∏_{z:V} z ∈ \{f(x)␣|␣x∈a\} ≃ \big\|∑_{x:\El␣a}f␣x=z \;\big\|_k\)\end{definition}

\textbf{Remark:} Having \(k\)-replacement is to have \(k\)-unordered
\(I\)-tupling for \(I = \El a\), for any \(a:V\).

In Section \ref{representations-of-types-in-e-structures} we further
explore the implications of \(k\)-replacement and how it affects
representations of types within ∈-structures.

\hypertarget{union}{%
\subsubsection{Union}\label{union}}

Just as for tupling, there is a notion of union to consider for each
type level, for ∈-structures, distinguished by the truncation of the
existential quantifier.

\begin{definition}[\agdalink{https://elisabeth.stenholm.one/univalent-material-set-theory/e-structure.property.union.html\#1110}]

An ∈-structure, \((V,∈)\), has \textbf{\((k+1)\)-union}, for
\(k : \Nat^∞_{-1}\), if for every \(x : V\) there is \(⋃_k x : V\) such
that
\(∏_{z:V} z ∈ ⋃_k x ≃ \big\|∑_{y:V} \; z ∈ y\ ×\ y ∈ x \; \big\|_k\).\end{definition}

As usual, from pairing and general unions, one can define binary
\(k\)-union by \(x ∪_k y ≔ ⋃_k \{x,y\}_k\). For \(z : V\) we have:
\begin{align}
    z ∈ x ∪_k y
        &≃ \Big\|∑_{w:V} \; z ∈ w\ ×\ w ∈ \{x,y\}_k\Big\|_{k-1} \\
        &≃ \Big\|∑_{w:V} \; z ∈ w\ ×\ \| w = x + w = y \|_{k-1}\Big\|_{k-1} \\
        &≃ \Big\|∑_{w:V} \; z ∈ w\ ×\ \left(w = x + w = y\right) \Big\|_{k-1} \\
        &≃ \Big\|z ∈ x + z ∈ y \Big\|_{k-1}
\end{align} Binary \(0\)-union is the usual binary union in set theory,
where copies are discarded. The higher level binary union keeps copies.
For example, \(\{∅\}_1 ∪_0 \{∅\}_1 = \{∅\}_0\), while
\(\{∅\}_1 ∪_k \{∅\}_1 = \{∅,∅\}_k\), for \(k ≥ 1\).

\hypertarget{exponentiation}{%
\subsubsection{Exponentiation}\label{exponentiation}}

Exponentiation states that for any two sets \(a\) and \(b\) there is a
set that contains exactly the functions from \(a\) to \(b\). In order to
express this property we need to first generalise the notion of a
function internal to an ∈-structure. As opposed to the other properties,
the notion of a function can be expressed in a uniform way for
\(∈\)-structures of any level. As in usual set theory, the notion of a
function is relative to a choice of ordered pairing structure.

But what constitutes a function between two generalised sets, say
\(a,b:V\)? The perhaps easiest answer, which is the one we will argue
for here, is that it is in essence a function \(\El␣a → \El␣b\), which
can then be represented as set itself by taking its graph. At level 0,
in usual set theory, a function is completely determined by its graph.
The general situation is, however, that if \(a\) and \(b\) are of level
\(n+1\) there is an \(n\)-type of function structures which can be put
on a set of pairs. We call this structure \(\oper_{a␣b}\), and
Proposition \ref{operation-emb-functions-equiv} constructs an
equivalence between \(\El␣a → \El␣b\) and sets with
\(\oper_{a␣b}\)-structure.

\textbf{Remark:} As usual in constructive set theory, we use the
exponentiation axiom instead of the powerset axiom. One could define the
powerset axiom for higher level structures by saying that \(z : V\) is a
subset of \(x : V\) if for all \(y : V\) there is an embedding
\(y ∈ z ↪ y ∈ x\). However, in any fixed-point model, the powerset axiom
would require a small subobject classifier in the ambient type theory,
hence why we do not consider it further in this paper.

\begin{definition}[\agdalink{https://elisabeth.stenholm.one/univalent-material-set-theory/e-structure.property.exponentiation.html\#2233}]

Given three elements \(a,b,f:V\) in an ∈-structure \((V,∈)\) with
ordered pairing \(〈-,-〉 : V × V ↪ V\), define a type: \begin{align*}
\oper_{a␣b}␣f ≔&\left(∏_{x:V}␣x ∈ a ≃ ∑_{y:V} 〈x , y〉 ∈ f \right) \\
             × &\left(∏_{x:V}∏_{y:V}〈x,y〉∈f→y∈b \right) \\
             × &\left(∏_{z:V}␣z ∈ f → ∑_{x:V}␣∑_{y:V}␣z=〈x , y〉 \right)
\end{align*}\end{definition}

The first conjunct of \(\oper_{a␣b}␣f\) states that the pairs in \(f\)
form an operation with domain \(a\), and the second states that its
codomain is \(b\). The third conjunct, which is always a proposition by
Proposition \ref{ordered-pair-prop}, states that \(f\) only contains
pairs.

Note that \(\oper\) here is the same as Definition 8 in
\cite{gylterud-multisets}. However, in this paper we define another type
which we show is equivalent to \(\oper\) and which is of a more type
theoretic flavor.

The type \(\oper_{a␣b}␣f\) is not always a proposition. Hence, an
element \(p : \oper_{a␣b}␣f\) should be regarded as an operation
structure on \(f\) with domain \(a\) and codomain \(b\). However, in the
case when \(a\) and \(b\) are mere sets, \(\oper_{a␣b}␣f\) is a
proposition.

\begin{proposition}[\agdalink{https://elisabeth.stenholm.one/univalent-material-set-theory/e-structure.property.exponentiation.html\#15763}]

Given three elements \(a,b,f:V\) in an ∈-structure \((V,∈)\) with
ordered pairing \(〈-,-〉 : V × V ↪ V\), if \(a\) and \(b\) are mere
sets in \((V,∈)\), then \(\oper_{a␣b}␣f\) is a proposition equivalent to
the following type: \begin{align*}
    &\left(∏_{x:V}␣x ∈ a → \exuniquedisp{y:V} 〈x , y〉 ∈ f \right) \\
    × &\left(∏_{x:V}∏_{y:V}〈x,y〉∈f\ →\ x∈a\ ×\ y∈b \right) \\
    × &\left(∏_{z:V}␣z ∈ f → ∑_{x:V}␣∑_{y:V}␣z=〈x , y〉 \right)
\end{align*}\end{proposition}

\begin{proof}

We start by proving the equivalence. First, we observe that for any
\(x : V\), \(p : x ∈ a\), and function
\(e : ∑_{y:V} 〈x , y〉 ∈ f → x ∈ a\) we have an equivalence
\begin{align*}
\fib␣e\ p\ ≃ ∑_{y:V} 〈x , y〉 ∈ f
\end{align*}

since for any \(q : ∑_{y:V} 〈x , y〉 ∈ f\) the type \(e␣q = p\) is
contractible because \(a\) is a mere set. It thus follows that for any
\(x:V\) we have \begin{align*}
\left(x ∈ a ≃ ∑_{y:V} 〈x , y〉 ∈ f\right)
    &≃ ∑_{e:∑_{y:V} 〈x , y〉 ∈ f → x ∈ a} ∏_{p:x ∈ a} \iscontr (\fib␣e␣p) \\
    &≃ \left(∑_{y:V} 〈x , y〉 ∈ f → x ∈ a\right)
        × \left(x ∈ a → \exuniquedisp{y:V} 〈x , y〉 ∈ f\right)
\end{align*}

The desired equivalence follows from this, after some currying and
rearranging.

Since \(\exuniqueinl{y:V} 〈x , y〉 ∈ f\) is a proposition it follows
that \(\left(∏_{x:V}␣x ∈ a → \exuniqueinl{y:V} 〈x , y〉 ∈ f \right)\)
is a proposition. When \(a\) and \(b\) are mere sets, \(x∈a\ ×\ y∈b\) is
a proposition, and hence
\(\\\left(∏_{x:V}∏_{y:V}〈x,y〉∈f\ →\ (x∈a)\ ×\ (y∈b) \right)\) is a
proposition. Finally, since ordered pairing is an embedding, and
\(∑_{x:V}␣∑_{y:V}␣z=〈x , y〉\) is essentially the fibres of the pairing
operation, it follows that
\[∏_{z:V}␣z ∈ f → ∑_{x:V}␣∑_{y:V}␣z=〈x , y〉\] is a
proposition.\end{proof}

This equivalence was proven as Lemma 6.10 in \cite{gylterud-iterative}
for the type \(V\) constructed there. Here we show the equivalence for
general ∈-structures.

\begin{corollary}

If \((V,∈)\) is of level 0 then \(\oper_{a␣b}␣f\) is a proposition, for
all \(a,b,f : V\).\end{corollary}

We can now define exponentiation for ∈-structures.

\begin{definition}[\agdalink{https://elisabeth.stenholm.one/univalent-material-set-theory/e-structure.property.exponentiation.html\#2442}]

An ∈-structure, \((V,∈)\) with an ordered pairing structure, has
\textbf{exponentiation} if for every two \(a,b:V\) there is an element
\(b^a\) such that \(∏_{f:V} (f ∈ b^a) ≃ \oper_{a␣b}␣f\).\end{definition}

While \(\oper\) is a straightforward internalisation of the set
theoretic definition of a function, it can be inconvenient to work with.
What follows is a definition that is equivalent to \(\oper\) but which
is sometimes easier to use.

\begin{definition}[\agdalink{https://elisabeth.stenholm.one/univalent-material-set-theory/e-structure.property.exponentiation.html\#2647}]

Given three elements \(a,b,f:V\) in an ∈-structure \((V,∈)\) with
ordered pairing \(〈-,-〉 : V × V ↪ V\), define a type: \begin{align*}
\operr_{a␣b}␣f ≔&∑_{φ:\El␣a␣→␣\El␣b}\ ∏_{z:V} \left(z ∈ f ≃ ∑_{x:\El␣a} 〈π_0\ x, π_0␣(φ\ x)〉=z\right)
\end{align*}\end{definition}

\begin{proposition}[\agdalink{https://elisabeth.stenholm.one/univalent-material-set-theory/e-structure.property.exponentiation.html\#2916}]

\label{operation-equiv}For any \(a,b,f:V\) we have the following
equivalence:

\begin{align*}
\oper_{a␣b}␣f ≃ \operr_{a␣b}␣f
\end{align*}\end{proposition}

Before we prove this equivalence we prove two general lemmas that we
will need, concerning fiberwise functions and equivalences.

\begin{lemma}[\agdalink{https://elisabeth.stenholm.one/univalent-material-set-theory/foundation.embeddings.html\#8702}]
\label{sigma-emb}

Let \(A\) and \(B\) be types and let \(C : B → \Type\) be a type family
over \(B\). For any embedding \(f : A ↪ B\) and element
\(γ : ∏_{(b,c):∑_{b:B} C␣b} \fib␣f\
b\) we have an equivalence \begin{align*}
∑_{a:A} C␣(f\ a) ≃ ∑_{b:B} C␣b.
\end{align*}\end{lemma}

\begin{proof}

Let \(F : ∑_{a:A} C␣(f\ a) → ∑_{b:B} C␣b\) be the function given by
\(F␣(a, c) := (f␣a, c)\). For \((b, c) : ∑_{b:B} C␣b\), we have the
following chain of equivalences: \begin{align}
    \fib␣F␣(b,c)
        &≃ ∑_{a : A} ∑_{c' : C␣(f␣a)} ∑_{p : f␣a = b} \tr{C}{p}{c'} = c \\
        &≃ ∑_{a : A} ∑_{p : f␣a = b} ∑_{c' : C␣(f␣a)} c' = \tr{C}{p^{-1}}{c} \\
        \label{sigma-emb-1}
        &≃ \fib␣f␣b
\end{align}

Step (\ref{sigma-emb-1}) uses the fact that
\(∑_{c' : C␣(f␣a)} c' = \tr{C}{p^{-1}}{c}\) is contractible.

Thus, \(\fib␣F\ (b,c)\) is a proposition, since \(f\) is an embedding,
which is inhabited by \(γ\), and therefore contractible.\end{proof}

There is an equivalence between fiberwise equivalences and equivalences
of the total spaces that respect the first coordinate. This will be
useful several times because it gives two equivalent characterisations
of equality in slices over a type.

\begin{lemma}[\agdalink{https://elisabeth.stenholm.one/univalent-material-set-theory/foundation.slice.html\#6332}]

\label{total-fiberwise-map}For any type \(A\), and any families
\(P, Q : A → \Type\) we have that

\begin{equation*}
    \left( ∏_{x:A} P␣x → Q␣x \right)
    ␣≃␣ \left( ∑_{α:∑_{x:A} P␣x → ∑_{x:A} Q␣x} π₀ ∘ α = π₀ \right)
\end{equation*}\end{lemma}

\begin{proof}

We have the following chain of equivalences: \begin{align}
    \label{fib-map-step-1}
    \left(∏_{x:A} P␣x → Q␣x\right)
        &≃ ∏_{(x,\_):∑_{x:A} P␣x} Q␣x \\
    \label{fib-map-step-2}
        &≃ ∏_{(x,\_):∑_{x:A} P␣x} \; ∑_{(x',\_):∑_{x':A} x' = x} Q␣x' \\
        &≃ ∏_{(x,\_):∑_{x:A} P␣x} \; ∑_{(x',\_):∑_{x':A} Q␣x'} x' = x \\
    \label{fib-map-step-4}
        &≃ ∑_{α:∑_{x:A} P␣x → ∑_{x:A} Q␣x} π₀ ∘ α = π₀
\end{align}

where (\ref{fib-map-step-1}) is currying, (\ref{fib-map-step-2}) follows
from the fact that \(∑_{x':A} x' = x\) is contractible and
(\ref{fib-map-step-4}) is the interchange law between Σ-types and
Π-types, together with function extensionality. The equivalence sends
\(f\) to \((\total␣f, \refl)\).\end{proof}

\textbf{Remark:} Lemma \ref{total-fiberwise-map} has been independently
added by Egbert Rijke to the agda-unimath library \cite{agda-unimath}.
The proof there is slightly different, with a direct construction of the
equivalence, instead of equivalence reasoning.

\begin{corollary}[\agdalink{https://elisabeth.stenholm.one/univalent-material-set-theory/foundation.slice.html\#8713}]

\label{total-fiberwise-equiv}For any type \(A\), and any families
\(P, Q : A → \Type\) we have an equivalence

\begin{equation*}
    \left( ∏_{x:A} P␣x␣≃␣Q␣x \right)
    ␣≃␣ \left( ∑_{α:∑_{x:A} P␣x␣≃␣∑_{x:A} Q␣x} π₀ ∘ α = π₀ \right)
\end{equation*}\end{corollary}

\begin{proof}

This follows from Lemma \ref{total-fiberwise-map} by an application of
Theorem 4.7.7 in the HoTT Book \cite[p.~185]{hottbook}, which states
that a fiberwise transformation is a fiberwise equivalence if and only
if the corresponding total function is an equivalence.\end{proof}

\begin{proof}[Proof of Proposition \ref{operation-equiv}]
We outline the key steps of the equivalence. For full details see the Agda
formalisation. We have the following chain of equivalences:

\begin{align}
\oper_{a␣b}␣f 
    \label{oper-equiv-1}
    ≃ &\left(∏_{z:\El␣f} \fib\ 〈-,-〉\ (π_0\ z)\right) \\
    × &\left(∑_{σ:\El␣a ≃ ∑_{(x,y):V × V} 〈x,y〉 ∈ f} π_0 ∼ π_0 ∘ π_0 ∘ σ\right) \nonumber \\
    × &\left(∑_{φ:∑_{(x,y):V × V} 〈x,y〉 ∈ f → \El␣b} π_0 ∘ φ ∼ π_1 ∘ π_0\right) \nonumber \\
    \label{oper-equiv-2}
    ≃ &∑_{γ:∏_{z:\El␣f} \fib\ 〈-,-〉\ (π_0\ z)} \\
        &\hspace{16pt} \left(∑_{σ:\El␣a ≃ \El␣f} π_0 ∘ σ^{-1} ∼ π_0 ∘ π_0 ∘ γ\right) \nonumber \\
        &\hspace{6pt} × \left(∑_{φ:\El␣f → \El␣b} π_0 ∘ φ ∼ π_1 ∘ π_0 ∘ γ\right) \nonumber \\
    \label{oper-equiv-3}
    ≃ &∑_{σ:\El␣a ≃ \El␣f} ∑_{φ:\El␣f → \El␣b} ∏_{z:\El␣f} \\
        &\hspace{16pt} ∑_{(x,y) : V × V}
            \left(〈x,y〉 = π_0␣z\right) × 
            \left(π_0␣(σ^{-1}\ z) = x\right) ×
            \left(π_0␣(φ\ z) = y\right) \nonumber \\
    \label{oper-equiv-4}
    ≃ &∑_{σ:\El␣a ≃ \El␣f} ∑_{φ:\El␣f → \El␣b} ∏_{z:\El␣f}
        〈π_0␣(σ^{-1}\ z),π_0␣(φ\ z)〉=π_0\ z \\
    \label{oper-equiv-5}
    ≃ &∑_{φ:\El␣a → \El␣b} ∑_{σ:\El␣a ≃ \El␣f} ∏_{x:\El␣a}
        〈π_0\ x, π_0␣(φ\ x)〉=π_0␣(σ\ x) \\
    \label{oper-equiv-6}
    ≃ &\ \operr_{a␣b}␣f
\end{align}

where in (\ref{oper-equiv-1}) we apply Corollary \ref{total-fiberwise-equiv}
to the first conjunct and Lemma \ref{total-fiberwise-map} to the second
conjunct, and then rearrange. In (\ref{oper-equiv-2}) we use Lemma
\ref{sigma-emb} to construct an equivalence $∑_{(x,y):V × V} 〈x,y〉 ∈ f ≃
\El␣f$ which we apply in the second and third conjuncts. In
(\ref{oper-equiv-3}) we rearrange and then use the interchange law for
$∏$-types and $∑$-types. In (\ref{oper-equiv-4}) we use extensionality for
cartesian products and the fact that
$∑_{(x,y):V × V} (π_0␣(σ^{-1}\ z),π_0␣(φ␣z))=(x,y)$ is contractible. In (\ref{oper-equiv-5}) we swap $\El␣a$ for
$\El␣f$ and then rearrange. Finally, in (\ref{oper-equiv-6}) we use
Corollary \ref{total-fiberwise-equiv} again.\qedhere

\end{proof}

The notion of \(\operr\) captures type theoretic functions in the sense
that the type of all operations from \(a : V\) to \(b : V\) is a subtype
of the type \(\El a → \El b\).

\begin{proposition}[\agdalink{https://elisabeth.stenholm.one/univalent-material-set-theory/e-structure.property.exponentiation.html\#20060}]

\label{operation-emb-functions}Given \(a, b : V\) in an ∈-structure
\((V,∈)\) with ordered pairing \(〈-,-〉\), there is a canonical
embedding \begin{align*}
    \left(∑_{f : V} \operr_{a␣b}␣f\right) ↪ \left(\El a → \El b\right)
\end{align*}\end{proposition}

\begin{proof}

By swapping the Σ-types we have the following equivalence

\begin{align}
    \left(∑_{f : V} \operr_{a␣b}␣f\right)
        \label{operation-emb-functions-step}
        ≃ ∑_{φ : \El a → \El b} ∑_{f : V} ∏_{z:V} \left(z ∈ f ≃ ∑_{x:\El␣a} 〈π_0\ x, π_0␣(φ\ x)〉=z\right)
\end{align}

The type \(∑_{f : V} ∏_{z:V} \left(z ∈ f ≃ ∑_{x:\El␣a} 〈π_0\ x, π_0␣(φ\
x)〉=z\right)\) is a proposition, by Proposition
\ref{is-prop-fixed-e-rel}. The embedding is thus the composition of
(\ref{operation-emb-functions-step}) with the first
projection.\end{proof}

An operation from \(a : V\) to \(b : V\) is thus a function from
\(\El a → \El b\). A natural question to ask is when all functions
\(\El a → \El b\) correspond to some operation.

\begin{proposition}[\agdalink{https://elisabeth.stenholm.one/univalent-material-set-theory/e-structure.property.exponentiation.html\#20601}]

\label{operation-emb-functions-equiv}Given \(a, b : V\) in an
∈-structure \((V,∈)\) of level \(n : \Nat^∞\), with ordered pairing
\(〈-,-〉\), the embedding in Proposition \ref{operation-emb-functions}
is an equivalence if \((V,∈)\) has \(n\)-replacement.\end{proposition}

\begin{proof}

Suppose \((V,∈)\) has \(n\)-replacement. By applying it to the map
\(λ␣x.〈π_0\ x, π_0␣(φ\ x)〉 : \El a → V\), we get an element of the
following type: \begin{align*}
    ∑_{f : V} ∏_{z:V} \left(z ∈ f ≃ \Big\|∑_{x:\El␣a} 〈π_0\ x, π_0␣(φ\ x)〉=z\Big\|_{n-1}\right)
\end{align*} \textbf{Claim:} The type
\(∑_{x:\El␣a} 〈π_0\ x, π_0␣(φ\ x)〉=z\) is \((n-1)\)-truncated.

Given that the claim is true we can drop the truncation. Thus the type
\begin{align*}
    ∑_{f : V} ∏_{z:V} \left(z ∈ f ≃ ∑_{x:\El␣a} 〈π_0\ x, π_0␣(φ\ x)〉=z\right)
\end{align*} is contractible, because it is an inhabited proposition,
and hence the embedding constructed in Proposition
\ref{operation-emb-functions} is an equivalence.

It remains to prove the claim. First, we have the following equivalence:
\begin{align}
    ∑_{x:\El␣a} \left(〈π_0\ x, π_0␣(φ\ x)〉=z\right)
        ≃ ∑_{(x , y) : \El a × \El b} \left(〈π_0\ x, π_0␣y〉=z\right)
            × \left(φ\ x = y\right)
\end{align} This follows from the fact that \(∑_{y : \El b} φ\ x = y\)
is contractible. The type \(φ\ x = y\) is \((n-1)\)-truncated since
\(V\) is an \(n\)-type, by Proposition \ref{m-h-level}. The type
\(∑_{(x , y) : \El a × \El b} \left(〈π_0\ x, π_0\ y〉=z\right)\) is the
fiber of the composite map \(〈-,-〉 ∘ (π_0 × π_0)\) over \(z : V\). The
map \(〈-,-〉\) is an embedding and thus an \((n-1)\)-truncated map. The
fibers of the two projection maps over \(z\) are \(z ∈ a\) and \(z ∈ b\)
respectively. These types are \((n-1)\)-truncated. So the composite map
is \((n-1)\)-truncated, and the claim follows.\end{proof}

\hypertarget{accessible-elements-and-foundation}{%
\subsubsection{Accessible elements and
foundation}\label{accessible-elements-and-foundation}}

We choose the same approach as the HoTT Book \cite{hottbook} to
well-foundedness, namely accessibility predicates. The axiom of
foundation is then the statement that all elements are accessible. This
need not be true in a given ∈-structure, but the subtype of accessible
elements inherits the ∈-structure from the base type. This new
∈-structure satisfies foundation.

\begin{definition}[\agdalink{https://elisabeth.stenholm.one/univalent-material-set-theory/order-theory.accessible-elements-relations.html\#1138}]

Given an ∈-structure, \((V,∈)\), define inductively the predicate
\(\Acc : V → \Type\) by

\begin{itemize}
\tightlist
\item
  \(\acc : ∏_{x : V} (∏_{y : V} y ∈ x → \Acc y) → \Acc x\)
\end{itemize}

\end{definition}

\begin{lemma}[\agdalink{https://elisabeth.stenholm.one/univalent-material-set-theory/order-theory.accessible-elements-relations.html\#3710}]

\label{lma-acc-prop}For every \(x : V\) the type \(\Acc x\) is a mere
proposition.\end{lemma}

\begin{proof}

Lemma 10.3.2 in the HoTT Book \cite[p.~454]{hottbook}.\end{proof}

\begin{definition}[\agdalink{https://elisabeth.stenholm.one/univalent-material-set-theory/e-structure.property.foundation.html\#283}]

An ∈-structure, \((V,∈)\), has \textbf{foundation} if
\(∏_{x:V}\Acc x\).\end{definition}

Since accessibility is a proposition, we can define for a given
∈-structure, the subtype of accessible elements.

\begin{definition}[\agdalink{https://elisabeth.stenholm.one/univalent-material-set-theory/e-structure.accessible-elements.html\#1416}]

Given an ∈-structure, \((V,∈)\), define the type
\(V_{\Acc} ≔ ∑_{x : V} \Acc x\) and define the binary relation
\(∈_{\Acc} ≔ λ␣(x : V_{\Acc})␣(y : V_{\Acc}).␣π_0␣x ∈ π_0␣y\).\end{definition}

The subtype of accessible elements inherits the ∈-structure from the
base type.

\begin{proposition}[\agdalink{https://elisabeth.stenholm.one/univalent-material-set-theory/e-structure.accessible-elements.html\#8984}]

Given an ∈-structure, \((V,∈)\), the pair \((V_{\Acc},∈_{\Acc})\) forms
an ∈-structure.\end{proposition}

\begin{proof}

We need to prove extensionality for \((V_{\Acc},∈_{\Acc})\). For
\(x,y : V_{\Acc}\) we have the following chain of equivalences:
\begin{align}
    \left(x = y\right)
        \label{subtype-e-str-1}
        &≃ \left(π_0␣x = π_0␣y\right) \\
        \label{subtype-e-str-2}
        &≃ ∏_{z:V} z ∈ π_0␣x ≃ z ∈ π_0␣y \\
        \label{subtype-e-str-3}
        &≃ ∑_{e\, :\, ∑_{z:V} z ∈ π_0␣x\ ≃\ ∑_{z:V} z ∈ π_0␣y}
            π_0 ∘ e ∼ π_0 \\
        \label{subtype-e-str-4}
        &≃ ∑_{e\, :\, ∑_{z:V} \left(z ∈ π_0␣x\right) × \left(\Acc z\right)
            \ ≃\ ∑_{z:V} \left(z ∈ π_0␣y\right) × \left(\Acc z\right)}
            π_0 ∘ e ∼ π_0 \\
        \label{subtype-e-str-5}
        &≃ ∑_{e\, :\, ∑_{z:V_{\Acc}} π_0␣z ∈ π_0␣x
            \ ≃\ ∑_{z:V_{\Acc}} π_0␣z ∈ π_0␣y} π_0 ∘ e ∼ π_0 \\
        &≃ ∏_{z:V_{\Acc}} z ∈_{\Acc} x ≃ z ∈_{\Acc} y
\end{align}

In step (\ref{subtype-e-str-1}) we use Lemma \ref{lma-acc-prop}. Step
(\ref{subtype-e-str-2}) is extensionality for \((V,∈)\) and step
(\ref{subtype-e-str-3}) is Corollary \ref{total-fiberwise-equiv}. In
step (\ref{subtype-e-str-4}) we use the fact that elements of accessible
sets are accessible and thus \(\Acc z\) is contractible. Step
(\ref{subtype-e-str-5}) is some rearranging of conjuncts in the base
type together with the characterisation of equality in subtypes for the
fibration. The last step is again an application of Corollary
\ref{total-fiberwise-equiv}.

Chasing \(\refl\) along this chain of equivalences we see that it is
sent to \(λ␣z.\idequiv\).\end{proof}

The subtype of accessible elements satisfies foundation.

\begin{theorem}[\agdalink{https://elisabeth.stenholm.one/univalent-material-set-theory/e-structure.accessible-elements.html\#9784}]

Given an ∈-structure, \((V,∈)\), the ∈-structure \((V_{\Acc},∈_{\Acc})\)
has foundation.\end{theorem}

\begin{proof}

There are two different accessibility predicates at play here. Let
\(\Acc\) and \(\acc\) be accessibility with respect to \((V,∈)\) and let
\(\Acc'\) and \(\acc'\) be accessibility with respect to
\((V_{\Acc},∈_{\Acc})\). We need to show \(∏_{x : V_{\Acc}} \Acc' x\),
which we do by using the induction principle of \(\Acc\) (modulo a
transport): \begin{align*}
    &α : ∏_{x : V_{\Acc}} \Acc' x \\
    &α(x,\acc␣f) ≔ \acc'␣\left(λ␣(y : V_{\Acc})(p : π_0␣y ∈ x).␣
        α(π_0␣y, f␣(π_0␣y)␣p)\right)\qedhere
\end{align*}\end{proof}

\hypertarget{internalisations-of-types-in--structures}{%
\section{\texorpdfstring{Internalisations of types in ∈-structures
\label{representations-of-types-in-e-structures}}{Internalisations of types in ∈-structures }}\label{internalisations-of-types-in--structures}}

Any element \(a:V\) gives rise to a type \(\El a\) (Definition
\ref{def-el}), and in some sense \(a\) represents \(\El a\) inside the
bigger structure of \((V,∈)\). For instance, an operation from \(a\) to
\(b\) is precisely a function from \(\El a → \El b\) (Proposition
\ref{operation-emb-functions}). A natural question to ask is: Which
types can be represented as elements in this way in a given ∈-structure?
In this section we introduce some basic vocabulary for talking about
this kind of representation. We apply this by giving a very flexible
formulation of the axiom of infinity -- which constructively is often
formulated as the existence of a set collecting the natural numbers for
some chosen encoding of these. The flexibility of our formulation is
that it takes the encoding as a parameter, making very few assumptions
about it. We prove that for ∈-structures that satisfy replacement, the
existence of a set of natural numbers is independent of encoding.

\hypertarget{internalisations-and-representations}{%
\subsection{Internalisations and
representations}\label{internalisations-and-representations}}

For the rest of this section, fix an ∈-structure \((V,∈)\) and
\(A : \Type\).

\begin{definition}[\agdalink{https://elisabeth.stenholm.one/univalent-material-set-theory/e-structure.internalisations.html\#1746}]

\textbf{A \((V,∈)\)-internalisation of \(A\)} is an element \(a : V\)
such that \(\El a ≃ A\).\end{definition}

There can be several different internalisations of \(A\). However, if we
fix an encoding of the elements of \(A\) in \(V\), there is at most one
element in \(V\) which is an internalisation of \(A\) with respect to
this encoding.

An encoding of the elements of \(A\) in \(V\) is a function \(A → V\).
We will call this a representation because of the superficial similarity
with classical representation theory, which can be seen as the study of
functors \(G → \operatorname{Vec}_k\) for some group \(G\) and field
\(k\). In our case representations are functions \(A → V\) where \(A\)
can be any type (not necessarily a group), and the codomain is some
∈-structure.

\begin{definition}[\agdalink{https://elisabeth.stenholm.one/univalent-material-set-theory/e-structure.internalisations.html\#1854}]

\textbf{A \((V,∈)\)-representation of \(A\)} is a function
\(A → V\).\end{definition}

Our main concern is going to be if a given representation gives rise to
an internalisation of the domain in the ∈-structure itself. In other
words, if the elements pointed out by the representation can be
collected to a set such that the type of elements is exactly the type
being represented.

\begin{definition}[\agdalink{https://elisabeth.stenholm.one/univalent-material-set-theory/e-structure.internalisations.html\#1929}]

Given a \((V,∈)\)-representation of \(A\), say \(f : A → V\), let
\textbf{an internalisation of \(f\)} be an \(a : V\) such that for every
\(z : V\) we have \(z ∈ a ≃ \fib f␣z\).\end{definition}

We say that a representation is internalisable if there is an
internalisation of it. An internalisation of a representation is an
internalisation of the domain.

\begin{proposition}[\agdalink{https://elisabeth.stenholm.one/univalent-material-set-theory/e-structure.internalisations.html\#2103}]

\label{int-type-from-int-repr}Let \(f : A → V\) be a
\((V,∈)\)-representation of \(A\) and let \(a : V\) be an
internalisation of \(f\). Then \(a\) is an internalisation of
\(A\).\end{proposition}

\begin{proof}

By Lemma 4.8.2 in the HoTT Book \cite{hottbook}, we have: \begin{align*}
    \El a ≡ \left(∑_{z : V} z ∈ a\right) ≃ \left(∑_{z : V} \fib f␣z\right) ≃ A,
\end{align*} showing that \(a\) is an internalisation of
\(A\).\end{proof}

Once a representation of a type is fixed, the internalisation is
uniquely defined. Thus, we will speak of \emph{the} internalisation of a
given representation.

\begin{proposition}[\agdalink{https://elisabeth.stenholm.one/univalent-material-set-theory/e-structure.internalisations.html\#2453}]

\label{is-prop-int-of-repr}Given a \((V,∈)\)-representation of \(A\),
say \(f : A → V\), the internalisation of \(f\) is uniquely defined.
That is, the type \(∑_{a:V} ∏_{z:V} \left( z∈a ≃ \fib f␣z \right)\) is a
proposition.\end{proposition}

\begin{proof}

Corollary of Proposition \ref{is-prop-fixed-e-rel}.\end{proof}

The notion of representation gives us a way to separate the
internalisation of a type into a structure part (how the elements are
encoded) and the property that this structure can be internalised. In
fact, an internalisation of a type is exactly the same as an
internalisation of a representation of that type.

\begin{proposition}[\agdalink{https://elisabeth.stenholm.one/univalent-material-set-theory/e-structure.internalisations.html\#2977}]

\label{equiv-int-repr}The type of all internalisable
\((V,∈)\)-representations of \(A\),
\(∑_{f:A→V}∑_{a:V} ∏_{z:V}z∈a ≃ \fib f␣z\) is equivalent to the type of
all internalisations of \(A\),
\(∑_{a:V}\left(\El␣a ≃ A\right)\).\end{proposition}

\begin{proof}

For \(a : A\), we have the following chain of equivalences:
\begin{align}
    \left(∑_{f : A → V} ∏_{z : V} z ∈ a ≃ \fib f␣z\right)
        \label{prop-int-repr-equiv-step-1}
        &≃ ∑_{f : A → V} ∑_{e : \El a ≃ A} f ∘ e = π_0 \\
        \label{prop-int-repr-equiv-step-2}
        &≃ ∑_{e : \El a ≃ A} ∑_{f : A → V} f = π_0 ∘ e^{-1} \\
        \label{prop-int-repr-equiv-step-3}
        &≃ \left(\El a ≃ A\right)
\end{align} Step (\ref{prop-int-repr-equiv-step-1}) is an application of
Corollary \ref{total-fiberwise-equiv}. We rearrange in step
(\ref{prop-int-repr-equiv-step-2}), and finally, step
(\ref{prop-int-repr-equiv-step-3}) is the fact that the total space of
paths from a given point is contractible.

The desired equivalence follows from the equivalence above, after some
rearranging.\end{proof}

If a representation \(f : A → V\) is an embedding, we say that the
representation is \emph{faithful}. The truncation level of a
representation determines the level of the internalisation. In the
special case of a faithful representation, the internalisation is a mere
set.

\begin{proposition}[\agdalink{https://elisabeth.stenholm.one/univalent-material-set-theory/e-structure.internalisations.html\#5394}]

An internalisation of \(f : A → V\) is a \((k+1)\)-type if and only if
\(f\) is a \(k\)-truncated representation, for
\(k : \Nat_{-1}\).\end{proposition}

\begin{proof}

Let \(a : V\) be an internalisation of \(f\). Then \(a\) is a
\((k+1)\)-type if and only if \(\fib f␣z\) is a \(k\)-type, for all
\(z : V\). In other words, \(a\) is a \((k+1)\)-type if and only if
\(f\) is a \(k\)-truncated map.\end{proof}

Note here that if \(f\) is a faithful representation, it is not
necessarily the case that the type \(A\) itself is a set. A mere set in
\(V\) can be the internalisation of a faithful representation of a type
of any level, as long as it can be embedded into \(V\). An exception to
this is if \(V\) is of level 0, meaning that it \emph{only} contains
mere sets, then \(V\) itself is a set, all internalisable
representations are faithful, and hence any type with an internalisable
representation will be a set.

Representations and internalisations give an new perspective on the
properties of Section \ref{section-e-structures}. All of them, except
foundation, can be seen as stating that a certain representation has an
internalisation. For instance, exponentiation can be seen as stating
that the representation sending a map \(\El␣a → \El␣b\) to its graph,
can be internalised.

Replacement, in particular, can be seen in an interesting light. In
classical set theory, replacement is the axiom schema which says that if
you have a set you can replace its elements and still get a set, as long
as the replacements are uniquely defined (by a formula in the
first-order language of set theory). In terms of representations we can
view this as follows: Given an internalisation of a type as a set, any
other representation can be internalised by replacing the elements of
the internalisation by their alternate representation. The various
levels of replacement properties correspond to restrictions on what kind
of representations can be replaced. For instance, 0-replacement gives
the replacement property for faithful representations.

\begin{proposition}[\agdalink{https://elisabeth.stenholm.one/univalent-material-set-theory/e-structure.internalisations.html\#7765}]

\label{internalisation-from-replacement}If \((V,∈)\) satisfies
\((k+1)\)-replacement and \(A\) has an internalisation, then any
\(k\)-truncated representation \(f : A → V\) of \(A\) has an
internalisation.\end{proposition}

\begin{proof}

Let \(a : V\) be an internalisation of \(A\), with a given
\(α : \El a ≃ A\). Then apply \((k+1)\)-replacement to
\(f∘α : \El a → V\) to obtain \(b : V\) such that for all \(z : V\)
there is an equivalence \begin{align}
    z ∈ b ≃ \Big\|∑_{x:\El␣a}f (α␣x) = z\Big\|_k
\end{align} Since \(f\)~is \(k\)-truncated and α is an equivalence, we
can swap the base type \(\El␣a\) with \(A\) and drop the truncation,
giving us an equivalence \begin{align}
    z ∈ b ≃ ∑_{x:A} f␣x = z
\end{align} Thus \(b\) internalises \(f\).\end{proof}

\hypertarget{natural-numbers-infinity}{%
\subsection{Natural Numbers / Infinity}\label{natural-numbers-infinity}}

There are many ways of formulating the axiom of infinity. The simplest
formulation is perhaps \(∃␣u␣(∅ ∈ u ∧ ∀␣x∈u\ ∃␣y∈u\ x∈y)\), which can be
found in some texts, such as \emph{Set Theory} by Bell \cite{bell2011}.
This formulation depends on foundation for the result to actually be an
infinite set, as this axiom is also satisfied by the co-hereditarily
finite sets (the final coalgebra of the the finite powerset functor), by
the element defined by the equation \(x = \{∅,x\}\). Another approach is
to define the notion of a successor set, for instance
\(s␣x = x ∪ \{x\}\) or \(s␣x = \{x\}\) and then postulate the existence
of a set containing \(∅\) and which is closed under \(s\). This becomes
an infinite set even without foundation, since the successor function
preserves well-foundedness and increases rank.

In Aczel's CZF \cite{aczel1978}, the axiom of infinity is formulated as
the existence of a set of natural numbers,
\(∃␣z␣\operatorname{Nat}(z)\). Aczel's formulation of this axiom
determines the set of natural numbers uniquely if foundation is assumed.
However, if foundation is not assumed, one may have fixed points of the
successor function which can be thrown in while still satisfying Aczel's
\(\operatorname{Nat}\) predicate. For instance, given a quine atom,
\(q = \{q\}\), then for any \(n\) we have
\(\operatorname{Nat}(n) → \operatorname{Nat}(n ∪ q)\). This can be
remedied by further assuming that \(n\) is accessible, in which case
\(n\) is uniquely determined.

When choosing a property corresponding to the axiom infinity for
∈-structures in general, we will leverage the fact that we are not bound
by first-order logic, and try to give a direct and intuitive
formulation, which does not depend on foundation or assumptions about
accessibility. We will also keep to the principle that properties
postulating existence of sets should be uniquely determined, or be
explicitly given as extra structure. Therefore, we will \emph{not} say
that an ∈-structure has natural numbers to mean that \(\Nat\) has an
internalisation, as the type \(∑_{a : V} \El a ≃ \Nat\) is not a
proposition. Having an internalisation of a type is a structure, rather
than a property. However, having an internalisation \emph{with respect
to a fixed representation} is a property (Proposition
\ref{is-prop-int-of-repr}), and it is the one we will use.

We will follow Aczel in choosing the natural numbers as the canonical
infinite set, but leave the exact encoding of the natural numbers as
extra structure. This leaves room for some exotic representations of the
natural numbers, as the internalisation might not even be a mere set.
Once an encoding is given, the set of naturals is uniquely determined,
and will from a certain point of view behave like the usual natural
numbers.

\begin{definition}[\agdalink{https://elisabeth.stenholm.one/univalent-material-set-theory/e-structure.property.natural-numbers.html\#1491}]

\label{natural-numbers}Given an ∈-structure \((V,∈)\) with a
representation of \(\Nat\), \(f : \Nat → V\), we say that \((V,∈)\) has
\textbf{natural numbers represented by \(f\)} if \(f\) has an
internalisation.\end{definition}

Suppose \(n\) is a set of natural numbers represented by \(f\) and let
\(e : \El n ≃ \Nat\) be the equivalence given by Proposition
\ref{int-type-from-int-repr}. Then we can define zero as \begin{align}
    &\zero : \El n \\
    &\zero ≔ e^{-1}␣0
\end{align} and we can define the successor function as \begin{align*}
    &\suc : \El n → \El n \\
    &\suc x = e^{-1} ∘ \sucN ∘ e
\end{align*} where \(\sucN\) is the successor function on \(\Nat\).

The usual induction principle holds, with respect to \(\zero\) and
\(\suc\).

\begin{proposition}[\agdalink{https://elisabeth.stenholm.one/univalent-material-set-theory/e-structure.property.natural-numbers.html\#2516}]

Given an ∈-structure \((V,∈)\) with natural numbers \(n : V\)
represented by \(f : \Nat → V\), let \(P : \El n → \Type\) be a type
family on \(\El n\). Given \(P␣\zero\) and
\(∏_{x : \El n} P␣x → P␣(\suc␣x)\), there is an element of the type
\(∏_{x : \El n} P␣x\).\end{proposition}

\begin{proof}

Let \(e : \El n ≃ \Nat\) be the equivalence given by Proposition
\ref{int-type-from-int-repr}. The result follows from the induction
principle on \(\Nat\), transported along \(e\).\end{proof}

\begin{example}[\agdalink{https://elisabeth.stenholm.one/univalent-material-set-theory/e-structure.property.natural-numbers.html\#8455}]

\label{von-neumann-nats}Let \((V,∈)\) be an ∈-structure with the empty
set, 0-singletons and binary 0-union. Then we can define the usual von
Neumann representation of the natural numbers \(f : \Nat → V\)
recursively by \(f␣0 ≔  ∅\) and \(f␣(k+1) ≔ f␣k␣∪₀␣\{f␣k\}₀\). This will
be a faithful representation, and an ∈-structure having natural numbers
represented by \(f\) will be equivalent to the usual characterisations
of the von Neumann natural numbers in set theory, such as the axiom of
infinity in CZF (in the well-founded sets).\end{example}

\begin{example}

Let \((V,∈)\) be an ∈-structure with the empty set, binary 1-union and
0-singletons, then let \(f : \Nat → V\) be defined by \(f␣0 ≔ ∅\) and
\(f␣(k+1)  ≔ f␣k␣∪₁␣\{∅\}₀\). This representation of \(\Nat\) is not
faithful, since \(f␣k\) has non-trivial automorphisms when k
\textgreater{} 1. For instance \(f␣3 =  \{∅,∅,∅\}₁\) has \(3! = 6\)
automorphisms. An internalisation, \(n : V\), of this representation
would be a multiset with the interesting property that
\((f␣k ∈ n) ≃ \Fin (k!)\).\end{example}

Interestingly, as long as there is enough replacement in the
∈-structure, the exact choice of representation of \(\Nat\) does not
matter:

\begin{proposition}[\agdalink{https://elisabeth.stenholm.one/univalent-material-set-theory/e-structure.property.natural-numbers.html\#1614}]

If \((V,∈)\) satisfies \((k+1)\)-replacement and has natural numbers for
some representation, then \((V,∈)\) has natural numbers represented by
\(f\) for any \(k\)-truncated representation \(f\).\end{proposition}

\begin{proof}

Suppose \((V,∈)\) has natural numbers for some representation. By
Proposition \ref{int-type-from-int-repr} it follows that \(\Nat\) has an
internalisation. Thus the result follows by Proposition
\ref{internalisation-from-replacement}.\end{proof}

\hypertarget{equivalence-of-extensional-coalgebras-and-u-like--structures}{%
\section{Equivalence of extensional coalgebras and U-like
∈-structures}\label{equivalence-of-extensional-coalgebras-and-u-like--structures}}

\label{equivalence-of-extensional-coalgebras-and-u-like-e-structures}

There is a well-established coalgebraic reading of set theory in which
models of set theory can be understood as coalgebras for the powerset
functor on the category of classes. This can be traced back to Rieger
\cite{rieger1957}, but perhaps found more explicitly in work by Osius
\cite{osius1974}.\footnote{More recent work in this direction is the
  work of Paul Taylor \cite{taylor2023} studying coalgebraic notions of
  well-foundedness and recursion.} In this section we establish a
similar correspondence between ∈-structures and coalgebras. We define a
hierarchy of functors \(\Tⁿ_U\), relative to a universe, and stratified
by type levels, whose extensional coalgebras correspond to ∈-structures
at the given type level.

\begin{definition}[\agdalink{https://elisabeth.stenholm.one/univalent-material-set-theory/functor.n-slice.html\#1463}]

For \(n : \Nat_{-1}^∞\), define \(\T_U^{n+1} : \Type → \Type\) by
\(\T_U^{n+1}␣X = ∑_{A:U} A ↪_n X\).\end{definition}

\textbf{Remark:} \(\T^∞_U\) is a polynomial functor, but for finite
\(n\), \(\Tⁿ_U\) is \emph{not} polynomial. Note that \(\T⁰_U\) is a
\(U\)-restricted `powerset functor' in that \(\T⁰_U␣X\) is the type of
\(U\)-small subtypes of \(X\).

The functorial action of \(\Tⁿ_U\) on a function \(f : X → Y\) is to
postcompose with \(f\) and then take the \((n-1)\)-image and the
\((n-1)\)-image inclusion. There is a size issue with this though. With
the usual construction of the \(n\)-image as
\(∑_{b:B} \|∑_{a:A}g␣a = b\|_n\) \cite{hottbook}, this type is in \(U\)
if \(U\) is closed under truncation and if both the domain and the
codomain are in \(U\). In our case the domain is in \(U\), but the
codomain need not be. Is it possible to weaken the assumption that the
codomain is in \(U\) (given that the domain is) and still conclude that
the \(n\)-image is in \(U\)?

\hypertarget{images-of-small-types}{%
\subsection{Images of small types}\label{images-of-small-types}}

As a higher groupoid, the image of a function \(f : X → Y\) looks like
it has the points from \(X\), but the paths are the paths from \(Y\)
resulting from applying \(f\). For example, consider the function
\(f : \unittype → \Type\) given by \(f␣x ≔ \twoelemtype\). The image of
\(f\) is the type \(∑_{x:\Type} \|\twoelemtype = x\|_{-1}\). There is
one point, namely \((\twoelemtype, |\refl|_{-1})\). However, by
univalence, there are two distinct paths
\((\twoelemtype, |\refl|_{-1}) = (\twoelemtype, |\refl|_{-1})\): the
identity equivalence and the equivalence that flips the elements in
\(\twoelemtype\) (we only need to consider equality on the first
coordinate).

More generally, the \(n\)-image looks like it has the \(k\)-cells, for
\(k ≤ n+1\), from \(X\), but the \((n+2)\)-cells are the \((n+2)\)-cells
from \(Y\) between the \((n+1)\)-cells resulting from applying \(f\).
\emph{From this intuition it seems reasonable that if all the
\((n+2)\)-iterated identity types of \(Y\) lie in \(U\), along with
\(X\), then the resulting higher groupoid \(n\)-image also lies in
\(U\).}

It turns out that it is enough to assume that this holds for
\((-1)\)-images. As pointed out below, this is not true for all
univalent universes, but the results presented here will rely upon the
assumption that this holds for our particular universe.

Rijke's modified join construction \cite{rijke2017} proves that a
sufficient criterion for this smallness of images assumption to hold in
a univalent universe is that \(U\) is closed under homotopy colimits (or
has graph quotients), and that global function extensionality holds.
Here, we have chosen instead to directly assume that images of
\(U\)-small types into locally \(U\)-small types are \(U\)-small.

The formulation of this assumption relies on two natural notions of
smallness --- which are also found in Rijke's work \cite{rijke2017}
along with the fact that being small a proposition. We generalise to a
definition of local smallness at any level.

\begin{definition}[\agdalink{https://elisabeth.stenholm.one/univalent-material-set-theory/foundation-core.small-types.html\#1212}]

A type \(X\) is \textbf{essentially \(U\)-small} if there is \(A : U\)
such that \(A ≃ X\). That is, \(∑_{A:U} A ≃ X\).\end{definition}

We proceed by induction to define local smallness at higher levels.

\begin{definition}[\agdalink{https://elisabeth.stenholm.one/univalent-material-set-theory/image-factorisation.html\#1703}]

Let \(X\) be a type. We define the notion of \(X\) being \(n\)-locally
\(U\)-small, for \(n : ℕ^{∞}\), as follows:

\begin{itemize}
    \item $X$ is $0$-locally $U$-small if it is essentially $U$-small.
    \item $X$ is $(n+1)$-locally $U$-small if for all $x, y : X$, the
    identity type $x = y$ is $n$-locally $U$-small.
    \item $X$ is ∞-locally $U$-small.
\end{itemize}

\end{definition}

We say that a \(1\)-locally \(U\)-small type is \textbf{locally
\(U\)-small}.

\begin{lemma}[\agdalink{https://elisabeth.stenholm.one/univalent-material-set-theory/foundation-core.small-types.html\#2464}]
\label{smallprop}

Being essentially \(U\)-small is a mere proposition.\end{lemma}

\begin{proof}

We must show that the type \(∑_{A:U} A ≃ X\) is a mere proposition. This
type looks close to being of the form \(∑_{b:T} b = a\), which is known
to be contractible, but since \(X\) is not a type in the universe we
cannot directly apply univalence. Instead we give a direct proof based
on the definition of being a mere proposition.

Let \((A,α)\) and \((B,β)\) be two elements of \(∑_{A:U} A ≃ X\).
Applying univalence, we observe that \(\ua␣(β^{-1} · α) : A = B\).

It remains to show that we get β by transporting α along the path
\(\ua␣(β^{-1} · α)\), in the family \(λ(Y : U). Y ≃ X\). But this is
easily computed with path algebra: \begin{align*}
\tr{λ Y. Y ≃ X}{\ua␣(β^{-1} · α)}{α}    &= α · (β^{-1} · α)^{-1} \\
                                        &= α · α^{-1} · β \\
                                        &= β \qedhere
\end{align*}\end{proof}

Now, the following assumption, which will be assumed throughout the rest
of the paper, can be formulated:

\begin{assumption}[Images of small types \agdalink{https://elisabeth.stenholm.one/univalent-material-set-theory/image-factorisation.html\#4904}]
\label{image-assumption}

For every small type \(A : U\) and every locally \(U\)-small
\(X : Type\), and given a function \(f : A → X\), there is a \(U\)-small
type \(\image␣f : U\), and functions \(\surj f : A ↠ \image␣f\) and
\(\incl f : \image␣f ↪ X\), such that for every \(a : A\) we have
\(\incl f\ (\surj f␣x) ≡ f␣x\).\end{assumption}

\textbf{Remark:} Not all univalent universes satisfy the assumption
above. For instance, given any univalent universe, \(U\), we can define
the subuniverse of mere sets, \(\Set_U\). This subuniverse is univalent,
and locally \(\Set_U\)-small, but the image of the map
\((λ␣x.\twoelemtype) : \unittype → \Set_U\) is not a set, as we saw
above. Thus, the image of this map is not essentially \(\Set_U\)-small.

While the work in this article is presented informally in type theory,
there might be some benefit here to spell out the assumption as
formulated in Agda. This formulation is uniform in universe levels (the
parameters \texttt{i} and \texttt{j}) so that it can be applied at any
level.

\begin{verbatim}
module _ {i j}
         {Domain : Type i} {Codomain : Type j}
         (_ : is-locally-small i Codomain)
         (f : Domain → Codomain) where

  postulate Image : Type i

  postulate image-inclusion : Image ↪ Codomain

  postulate image-quotient : Domain ↠ Image

  postulate image-β
     : ∀ x → (image-inclusion 〈 image-quotient 〈 x 〉 〉) ↦ f x

  {-# REWRITE image-β #-}
\end{verbatim}

This code can be found in the project source repository, where the file
containing the above snippet is called \texttt{image-factorisation.agda}
(\agdalink{https://elisabeth.stenholm.one/univalent-material-set-theory/image-factorisation.html}).

From the assumption of small \((-1)\)-images, the smallness of
\(n\)-images follows.

\begin{proposition}[\agdalink{https://elisabeth.stenholm.one/univalent-material-set-theory/image-factorisation.html\#7077}]

\label{small-n-images}Let \(n : \Nat^∞_{-2}\). For every small type
\(A : U\) and every \((n+2)\)-locally \(U\)-small \(X : \Type\), and
given a function \(f : A → X\), there is a small type
\(\image_n␣f : U\), together with an \(n\)-connected map
\(\surj_n f : A ↠_n \image_n␣f\) and an \(n\)-truncated map
\(\incl_n f : \image_n␣f ↪_n X\), such that for every \(a : A\) we have
\(\incl f\ (\surj f␣x) = f␣x\).\end{proposition}

\begin{proof}

This is Proposition 2.2 in \cite{christensen_non-accessible_2021}, but
we will also give our own proof here, formulated in a slightly different
way. Let \(A : U\) and \(X : Type\) be \(n\)-locally \(U\)-small, and
let \(f : A → X\). We proceed by induction on \(n\).

For the base case we need to construct a type \(\image_{-2}␣f : U\)
together with a \((-2)\)-connected map
\(\surj_{-2} f : A ↠_{-2} \image␣f\) and a \((-2)\)-truncated map
\(\incl_{-2} f : \image␣f ↪_{-2} X\). Note that any map is
\((-2)\)-connected and that \((-2)\)-truncated maps are precisely
equivalences. By assumption \(X\) is essentially \(U\)-small, i.e.~there
is a type \(X' : U\) together with an equivalence \(e : X ≃ X'\). So we
take \(\image_{-2}␣f ≔ X'\), \(\surj_{-2} f ≔ e ∘ f\) and
\(\incl_{-2} f ≔ e^{-1}\).

Suppose the proposition holds for \(n\), and suppose \(X\) is
\((n+1)\)-locally \(U\)-small. Let \(\image^{HoTT}_{n} g\) denote the
HoTT Book definition of the \(n\)-image of a function \(g\)
\cite[Def.~7.6.3]{hottbook}. We will show that \(\image^{HoTT}_{n-1} f\)
is essentially \(U\)-small. By Theorem 7.6.6 in \cite{hottbook} we have
a factorisation

\begin{center}
    \begin{tikzcd}
        A \arrow[rr, "f"] \arrow[rdd, "s"'] &                                                           & X \\
                                            &                                                           &   \\
                                            & \image^{HoTT}_{n-1} f \arrow[ruu, "i"'] &  
    \end{tikzcd}
\end{center}

where \(s\) is \((n-1)\)-connected and \(i\) is \((n-1)\)-truncated. The
map \(s\) is in particular surjective, so by the uniqueness of the image
factorisation, \(\image^{HoTT}_{n-1} f ≃ \image␣s\). Thus we are done if
we can show that \(\image^{HoTT}_{n-1} f\) is locally \(U\)-small. Since
being \(U\)-small is a proposition and \(s\) is surjective, it is enough
to show that \(s␣x = s␣y\) is essentially \(U\)-small for all
\(x,y : A\). But by the characterisation of equality in \(n\)-images we
have the equivalence
\(\left(s␣x = s␣y\right) ≃ \image^{HoTT}_{n-2}␣\left(\ap{f}␣x␣y\right)\).
The domain \(x = y\) of \(\ap{f}␣x␣y\) is in \(U\) and the codomain
\(f␣x = f␣y\) is \(n\)-locally \(U\)-small by assumption. Thus, by the
induction hypothesis and the uniqueness of the \((n-2)\)-image
factorisation,
\(\image^{HoTT}_{n-2}␣\left(\ap{f}␣x␣y\right) ≃ \image_{n-2}␣\left(\ap{f}␣x␣y\right)\)
and thus \(s␣x = s␣y\) is essentially \(U\)-small.

For the case \(n = ∞\), note that an ∞-connected map is an equivalence
and that any map is ∞-truncated. Thus we take \(\image_∞␣f ≔ A\),
\(\surj_∞␣f ≔ \idequiv\) and \(\incl_∞␣f ≔ f\).\end{proof}

\hypertarget{the-functor-tux207f_u}{%
\subsection{\texorpdfstring{The functor
\(\Tⁿ_U\)}{The functor \textbackslash Tⁿ\_U}}\label{the-functor-tux207f_u}}

With the construction of \(U\)-small \(n\)-images, we can define the
functorial action of \(\Tⁿ_U\). For \(n : \Nat^∞\), and two types
\(X, Y\) such that \(Y\) is \((n+1)\)-locally \(U\)-small, the
\((n-1)\)-image of any function \(φ : X → Y\), \(\image_{n-1}␣φ\), lies
in \(U\) by Proposition \ref{small-n-images}. Moreover, the image
inclusion \(\incl_{n-1}␣φ : \image_{n-1}␣φ → Y\) is an
\((n-1)\)-truncated map. Thus we make the following definition:

\begin{definition}[\agdalink{https://elisabeth.stenholm.one/univalent-material-set-theory/functor.n-slice.html\#4079}]

Let \(n : \Nat^∞\). For types \(X, Y\) such that \(Y\) is
\((n+1)\)-locally \(U\)-small, and a function \(φ : X → Y\) let
\(\Tⁿ_U␣φ : \Tⁿ_U␣X → \Tⁿ_U␣Y\) be the function \begin{align*}
        \Tⁿ_U φ\ (A , f) ≔ \left(\image_{n-1}␣(φ ∘ f), \incl_{n-1}␣(φ ∘ f)\right)
\end{align*}\end{definition}

\textbf{Note:} For \(n = ∞\) we get \(\T^∞_U φ\ (A, f) ≡ (A, φ ∘ f)\).

Thinking of the elements in \(\Tⁿ_U␣X\) as slices over the type \(X\),
equality should be fiberwise equivalence. This is indeed the case, as
the next proposition shows.

\begin{proposition}[\agdalink{https://elisabeth.stenholm.one/univalent-material-set-theory/functor.n-slice.html\#2527}]

\label{T-n-equality}Given \(n : \Nat^∞\) and a type \(X\), for
\((A,f), (B,g) : \Tⁿ_U␣X\) there is an equivalence \begin{equation*}
    \left((A,f) = (B,g)\right) ≃ ∏_{x : X} \fib␣f␣x ≃ \fib␣g␣x
\end{equation*}\end{proposition}

\begin{proof}

For \((A,f), (B,g) : \T^∞_U␣X\) we have the following chain of
equivalences: \begin{align}
    \left((A,f) = (B,g)\right)
        \label{T-n-equality-1}
        &≃ ∑_{α : A = B} g ∘ (\coe␣α) = f \\
        \label{T-n-equality-2}
        &≃ ∑_{α : A ≃ B} g ∘ α = f \\
        \label{T-n-equality-3}
        &≃ ∑_{α:∑_{x:X} \fib␣f␣x \; ≃ \; ∑_{x:X} \fib␣g␣x} π₀ ∘ α = π₀ \\
        \label{T-n-equality-4}
        &≃ ∏_{x : X} \fib␣f␣x ≃ \fib␣g␣x
\end{align}

Step (\ref{T-n-equality-1}) is the usual characterisation of equality in
φ-types, step (\ref{T-n-equality-2}) is the univalence axiom, step
(\ref{T-n-equality-3}) is Lemma 4.8.2 in the HoTT Book
\cite[p.~186]{hottbook} and step (\ref{T-n-equality-4}) is Corollary
\ref{total-fiberwise-equiv}. (Note: the equivalence
\(\left(∑_{α : A ≃ B} g ∘ α = f\right) ≃ ∏_{x:X}( \fib␣f␣x ≃ \fib␣g␣x)\)
was proved as Lemma 5 in \emph{Multisets in type theory}
\cite{gylterud-multisets}, by explicitly constructing an equivalence.)

For finite \(n\), the type \(\Tⁿ_U␣X\) is a subtype of \(\T^∞_U␣X\), and
thus has the same identity types.\end{proof}

The characterisation of equality on \(\Tⁿ_U\) allows us to compute the
type level.

\begin{corollary}[\agdalink{https://elisabeth.stenholm.one/univalent-material-set-theory/functor.n-slice.html\#2750}]

\label{T-n-level}Given \(n : \Nat^∞\) and a type \(X\), the type
\(\Tⁿ_U␣X\) is an \(n\)-type.\end{corollary}

\begin{proof}

For \((A,f), (B,g) : \Tⁿ_U␣X\) and \(x : X\), the types \(\fib␣f␣x\) and
\(\fib␣g␣x\) are \((n-1)\)-truncated. Thus the identity type
\((A,f) = (B,g)\) is \((n-1)\)-truncated.\end{proof}

The functor \(\Tⁿ_U\) preserves homotopies.

\begin{proposition}[\agdalink{https://elisabeth.stenholm.one/univalent-material-set-theory/functor.n-slice.html\#4270}]

\label{T-n-preserves-homotopies}Let \(n : \Nat^∞\). For types \(X, Y\)
such that \(Y\) is \((n+1)\)-locally \(U\)-small, and functions
\(φ, ψ : X → Y\), any homotopy \(ε  : φ ∼ ψ\) induces a homotopy
\(\Tⁿ_U␣ε : \Tⁿ_U φ ∼ \Tⁿ_U ψ\). Moreover, \(\reflhtpy\) is sent to
\(\reflhtpy\).\end{proposition}

\begin{proof}

Given \(ε : φ ∼ ψ\) we construct the homotopy \begin{align*}
      &\Tⁿ_U ε : \Tⁿ_U φ ∼ \Tⁿ_U ψ \\
      &\Tⁿ_U ε␣(A , f) ≔ \ap{λ␣σ.\Tⁿ_U σ␣(A , f)} (\funext␣ε)
 \end{align*} For \(\Tⁿ_U \reflhtpy\) we have the following chain of
equalities: \begin{align*}
      \Tⁿ_U \reflhtpy␣(A , f)
           &≡ \ap{λ␣σ.\Tⁿ_U σ␣(A , f)} (\funext␣\reflhtpy) \\
           &= \ap{λ␣σ.\Tⁿ_U σ␣(A , f)} \refl \\
           &≡ \refl \qedhere
 \end{align*}\end{proof}

We recap the definition of an algebra for the functor \(\Tⁿ_U\) and a
homomorphism between two algebras.

\begin{definition}[\agdalink{https://elisabeth.stenholm.one/univalent-material-set-theory/functor.n-slice.html\#12741}]

A \(\Tⁿ_U\)-algebra, \((X,m)\) consists of an \((n+1)\)-locally
\(U\)-small type \(X\) and a map \(m : \Tⁿ_U X → X\).\end{definition}

The restriction to \((n+1)\)-locally \(U\)-small carrier types for
\(\Tⁿ_U\) algebras, is because \(\Tⁿ_U\) is only functorial on such
types.

\begin{definition}[\agdalink{https://elisabeth.stenholm.one/univalent-material-set-theory/functor.n-slice.html\#13290}]

Given two \(\Tⁿ_U\) algebras, \((X,m)\) and \((X',m')\), a
\(\Tⁿ_U\)-algebra homomorphism from \((X,m)\) to \((X',m')\) is a pair
\((φ , α)\) where \(φ : X →  X'\) and
\(α : φ ∘ m ∼ m' ∘ \Tⁿ_U φ\).\end{definition}

For two \(\Tⁿ_U\)-algebra homomorphisms \((φ,α)\) and \((ψ,β)\) to be
equal, there needs to be a homotopy between the underlying maps φ and ψ.
But there is also a coherence condition on α and β, which can be seen as
saying that the two natural ways of proving that
\(m'␣(\Tⁿ_U␣φ) =_X' ψ␣(m␣a)\) are equal identifications in \(X'\). This
holds trivially if \(X'\) is a mere set, and in that case only the
homotopy between the maps is relevant.

\begin{lemma}[\agdalink{https://elisabeth.stenholm.one/univalent-material-set-theory/functor.n-slice.html\#13862}]

\label{lma-T-n-alg-hom-eq}For any two \(\Tⁿ_U\)-algebra homomorphisms
\((φ,α)\) and \((ψ,β)\), from \((X,m)\) to \((X',m')\), their identity
type is characterised by:

\begin{align}
     \left((φ , α) = (ψ, β)\right) 
      ≃ ∑_{ε : φ ∼ ψ} ∏_{a:\Tⁿ_U X} α␣a · \ap{m'} (\Tⁿ_U␣ε␣a) = ε␣(m␣a) · β␣a
 \end{align}\end{lemma}

\begin{proof}

We use the structure identity principle. Given an algebra homomorphism
\((φ  , α)\), we need to define families \(P : (X → X') → \Type\) and
\(Q : ∏_{ψ : X  → X'} \left(ψ ∘ m ∼ m' ∘ \Tⁿ_U ψ\right) → P␣ψ → \Type\),
which should characterise the identity type of the base type and of the
fibration respectively. Then we need to construct elements \(p : P␣φ\)
and \(q :  Q␣φ␣α␣p\). Finally, we need to show that for all
\(ψ : X → X'\), \(\left(φ =  ψ\right) ≃ P␣ψ\) and that for all
\(α' : φ ∘ m ∼ m' ∘ \Tⁿ_U φ\), \(\left(α =  α'\right) ≃ Q␣φ␣α'␣p\).

The family \(P\) is defined as \(P␣ψ ≔ \left(φ ∼ ψ\right)\), and \(Q\)
is defined as
\(Q␣ψ␣β␣p' ≔ ∏_{a:\Tⁿ_U X} α␣a · \ap{m'} (\Tⁿ_U␣p'␣a) = p'␣(m␣a) · β␣a\).
For the element in \(P␣φ\) we take \(p ≔ \reflhtpy\), and for
\(Q␣φ␣α␣p\) and \((A  , f) : \Tⁿ_U X\) we have the following chain of
equalities: \begin{align}
      α␣(A , f) · \ap{m'} (\Tⁿ_U␣p␣(A , f))
           \label{alg-hom-eq-1}
           &= α␣(A , f) · \ap{m'} \refl \\
           &= α␣(A , f) \\
           &= p␣(m␣(A , f)) · α␣(A , f)
 \end{align} In step (\ref{alg-hom-eq-1}) we use Proposition
\ref{T-n-preserves-homotopies}.

It remains to construct the equivalences. By function extensionality,
\(\left(φ = ψ\right) ≃ \left(φ ∼ ψ\right)\). Moreover, for
\(α' : φ ∘ m ∼ m' ∘ \Tⁿ_U φ\) we have the following chain of
equivalences: \begin{align}
      \left(α = α'\right)
           \label{alg-hom-eq-2}
           &≃ \left(α ∼ α'\right) \\
           &≃ ∏_{a : \Tⁿ_U X} α␣a · \ap{m'} \refl = \refl · α'␣a \\
           \label{alg-hom-eq-3}
           &≃ ∏_{a:\Tⁿ_U X} α␣a · \ap{m'} (\Tⁿ_U␣p␣a) = p␣(m␣a) · α'␣a \\
           &≡ Q␣φ␣α'␣p
 \end{align} Step (\ref{alg-hom-eq-2}) is function extensionality, and
in step (\ref{alg-hom-eq-3}) we again use Proposition
\ref{T-n-preserves-homotopies}.\end{proof}

\hypertarget{u-likeness}{%
\subsection{U-likeness}\label{u-likeness}}

In classical set theory, the Mostowsky collapse relates any well-founded
\emph{set-like} relation to the well-founded hierarchy. A set-like
relation is a class relation, say \(R\), where \(\{y \ |\ (x,y)∈R\}\) is
a set for all \(x\). The notion of ``being small'', i.e.~being a set, is
also one of the basic distinctions in algebraic set theory
\cite{joyal_algebraic_1995}.

In type theory, types take the role of classes while the universe \(U\)
provides a measure of smallness, akin to being a set in set theory.
Hence, the notion of a \(U\)-like ∈-structure will mirror the classical
notion of an extensional, set-like relation.

\begin{definition}[\agdalink{https://elisabeth.stenholm.one/univalent-material-set-theory/e-structure.u-like.html\#1927}]

An ∈-structure, \((V,∈)\) is \textbf{\(U\)-like} if \(\El a\) is
essentially \(U\)-small for every \(a : V\).\end{definition}

\textbf{Remark:} To simplify notation, we will coerce \(\El a : U\) in
\(U\)-like ∈-structures.

\begin{proposition}[\agdalink{https://elisabeth.stenholm.one/univalent-material-set-theory/e-structure.u-like.html\#2077}]

Being \(U\)-like is a mere proposition for any ∈-structure,
\((V,∈)\).\end{proposition}

The interaction between internalisable types and \(U\)-likeness is
straightforward:

\begin{proposition}[\agdalink{https://elisabeth.stenholm.one/univalent-material-set-theory/e-structure.internalisations.html\#6072}]

An ∈-structure, \((V,∈)\), is \(U\)-like if any only if for each
internalisable \(A : \Type\), the type \(A\) is essentially
\(U\)-small.\end{proposition}

\begin{proof}

We have the following chain of equivalences: \begin{align}
    ∏_{A : \Type} &\left(∑_{a : V} \El a ≃ A\right) → ∑_{X : U} A ≃ X \nonumber\\
    &≃ ∏_{a : V} \left(∑_{A : \Type} \El a ≃ A\right) → ∑_{X : U} \El a ≃ X \\
    \label{u-like-internalisation-step-1}
    &≃ ∏_{a : V} \left(∑_{X : U} \El a ≃ X\right)
\end{align} In step (\ref{u-like-internalisation-step-1}) we use the
fact that \(∑_{A : \Type} \El a ≃ A\) is contractible.\end{proof}

We will now investigate in detail the relationship between \(U\)-like
∈-structures and \(\Tⁿ_U\)-coalgebras.

\begin{lemma}[∈-structures are coalgebras \agdalink{https://elisabeth.stenholm.one/univalent-material-set-theory/e-structure.u-like.html\#7586}]
\label{coalgebra-structures}

For a fixed \(V\), having a \(U\)-like ∈-structure on \(V\) is
equivalent to having a coalgebra structure \(V ↪ \T^∞_U␣V\), which is an
embedding. This equivalence sends the relation \(∈␣: V → V → \Type\) to
the coalgebra \(λ␣a.(\El␣a, π_0)\) and, in the other direction, it sends
the coalgebra \(m :  V → \T^∞_U␣V\) to the relation
\(λ␣b␣a.\fib\ (π_1␣(m␣a))\ b\).\end{lemma}

\begin{proof}

We have the following chain of equivalences: \begin{align}
\label{coalgebra-structures-step-1}
V → ∑_{A:U}(A → V) 
    &≃ \left(V → ∑_{F : V → \Type} ∑_{E : U} \left(\left(∑_{b : V} F␣b\right)≃E \right)\right) \\
\label{coalgebra-structures-step-2}
    &≃ ∑_{∈ : V → V → \Type} ∏_{a:V} ∑_{E : U} \left(\left(∑_{b : V} b ∈ a\right)≃E\right)
\end{align}

The equivalence (\ref{coalgebra-structures-step-1}) follows by
substituting \(∑_{A:U}(A → V)\) with the equivalent type
\newline\(∑_{F : V → \Type} ∑_{E : U} (∑_{b : V} F␣b)≃E\), since a
fibration \(A → V\), with a total type \(A : U\) is equivalent to a
family over \(V\) with \(U\)-small total type, as being \(U\)-small is a
mere proposition. The equivalence (\ref{coalgebra-structures-step-2}) is
essentially currying.

Chasing a coalgebra, respectively an ∈-relation, along the chain of
equivalences we see that it computes as stated.

The type
\(∑_{∈ : V → V → \Type} ∏_{a:V} ∑_{E : U} \left(\left(∑_{b : V} b ∈ a\right)≃E\right)\)
is exactly the type of ∈-structures on \(V\) which are \(U\)-like,
except for the extensionality requirement. It remains to show that the
coalgebra is an embedding if and only if the corresponding
\(∈\)-relation, given by the equivalence above, is extensional.

Given a coalgebra \(m : V → \T^∞_U␣V\), let ∈ be the corresponding
relation given by the equivalence. By Proposition \ref{T-n-equality} we
have, for any \(x, y : V\), an equivalence \begin{align*}
e : \left(m␣x = m␣y\right) ≃ \left(∏_{z : V} z ∈ x ≃ z ∈ y\right)
\end{align*}

This equivalence sends \(\refl\) to \(λ␣z.\idequiv\). Let
\(e' : x = y → ∏_{z : V} z ∈ x ≃ z ∈ y\) be the map defined by path
induction as \(e'␣\refl ≔ λ␣z.\idequiv\). The following diagram
commutes:

\begin{center}
\begin{tikzcd}
    x = y \arrow[rr, "\ap{m}"] \arrow[rd, "e'"'] &  & m␣x = m␣y \arrow[ld, "e"] \\
        & ∏_{z : V} z ∈ x ≃ z ∈ y &
\end{tikzcd}
\end{center}

Since \(e\) is an equivalence it follows that \(\ap{m}\) is an
equivalence if and only if \(e'\) is an equivalence.\end{proof}

\begin{lemma}

Given a \(U\)-like ∈-structure, \((V,∈)\), and any \(n:\Nat_{-1}\), the
following are equivalent:

\begin{enumerate}
\def\labelenumi{\arabic{enumi}.}
\tightlist
\item
  \(b ∈ a\) is an \(n\)-type for all \(a,b:V\),
\item
  \(λ␣a.(\El␣a, π_0)\) restricts to a map \(V ↪ \T^{n+1}_U␣V\).
\end{enumerate}

\end{lemma}

\begin{proof}

The fibres of \(π₀ : (∑_{b : V} b ∈ a) → V\) are exactly the types
\(b ∈ a\).\end{proof}

The following theorem combines the two previous lemmas in order to
characterise \(U\)-like ∈-structures of various levels in terms of
coalgebras.

\begin{theorem}[\agdalink{https://elisabeth.stenholm.one/univalent-material-set-theory/e-structure.u-like.html\#8328}]

\label{U-like-equiv-T-n-coalgebra} For a fixed \(V\) and for
\(n : \Nat^∞_{-1}\), having a \(U\)-like ∈-structure on \(V\) such that
\(b ∈ a\) is an \(n\)-type for all \(a,b:V\) is equivalent to having a
coalgebra structure \(V ↪ \T^{n+1}_U␣V\).\end{theorem}

\begin{proof}

Simple corollary of the previous two lemmas.\end{proof}

\hypertarget{fixed-point-models}{%
\section{\texorpdfstring{Fixed-point models
\label{section-fixed-point-models}}{Fixed-point models }}\label{fixed-point-models}}

Rieger's theorem is a result in set theory which states that any class
with a set-like binary relation, which is a fixed-point of the related
powerset functor, is a model of ZFC\(^-\) (ZF without foundation)
\cite{rieger1957}. In this section we prove an analogous result,
considering the family of functors \(\Tⁿ_U\) as higher level
generalisations of the powerset functor. Any fixed-point for \(\Tⁿ_U\)
is in particular a \(\Tⁿ_U\)-coalgebra for which the coalgebra map is an
embedding. Thus, by Theorem \ref{U-like-equiv-T-n-coalgebra}, such a
fixed-point gives rise to a \(U\)-like ∈-structure. This induced
∈-structure will satisfy almost all the properties defined in Section
\ref{section-e-structures}, at some level, the only exception being
foundation.

So we assume in this section that we are given a type \(V\) and an
equivalence \(\sup : \Tⁿ_U V ≃ V\). Let \(\desup : V ≃ \Tⁿ_U V\) be the
inverse of \(\sup\). We will specifically show that the induced
∈-structure on \(V\) has the following properties:

\begin{itemize}
     \item Empty set.
     \item $U$-restricted $n$-separation.
     \item If $V$ is $(n+1)$-locally $U$-small, it has ∞-unordered $I$-tupling
          for all $(n-1)$-truncated types $I : U$.
     \item If $V$ is $(k+1)$-locally $U$-small, for some $k ≤ n$ then it has:
          \begin{itemize}
               \item $k$-unordered $I$-tupling for all $I : U$,
               \item $k$-replacement,
               \item $k$-union.
          \end{itemize}
     \item $V$ has exponentiation for all ordered pairing structures.
     \item $V$ has natural numbers represented by $f$ for any $(n-1)$-truncated
          representation $f : \Nat → V$.
 \end{itemize}

\textbf{Notation:} Given \(x : V\), we will use the notation
\(\overline{x} : U\) and \(\widetilde{x} : \overline{x} ↪_{n-1} V\) for
the type and \((n-1)\)-truncated map such that
\(x = \sup␣(\overline{x}, \widetilde{x})\),
i.e.~\(\overline{x} = π_0␣(\desup␣x)\) and
\(\widetilde{x} = π_1␣(\desup␣x)\).

\begin{definition}[\agdalink{https://elisabeth.stenholm.one/univalent-material-set-theory/e-structure.from-P-inf-coalgebra.html\#1051}]

Let \((V , ∈)\) be the ∈-structure on \(V\) given by Theorem
\ref{U-like-equiv-T-n-coalgebra}. Note that for \(x, y : V\),
\begin{align*}
      y ∈ x ≡ \fib␣\widetilde{x}␣y
 \end{align*}\end{definition}

\begin{proposition}[\agdalink{https://elisabeth.stenholm.one/univalent-material-set-theory/e-structure.from-P-inf-coalgebra.html\#1682}]

\label{index-type-lma}For any \(x:V\) we have that
\(\overline{x} ≃ \El␣x\).\end{proposition}

\begin{proof}

We have the following: \begin{align}
      \El␣x
           ≡ \left(∑_{y : V} \fib␣\widetilde{x}␣y\right)
           ≃ \overline{x}
 \end{align} where the equivalence in the last step is Lemma 4.8.2 in
the HoTT Book \cite[p.~142]{hottbook}.\end{proof}

\begin{proposition}[\agdalink{https://elisabeth.stenholm.one/univalent-material-set-theory/fixed-point.core.html\#1223}]
\label{V-is-n-type}

The type \(V\) is an \(n\)-type.\end{proposition}

\begin{proof}

\(V ≃ \Tⁿ_U V\) and \(\Tⁿ_U␣V\) is an \(n\)-type by Corollary
\ref{T-n-level}.\end{proof}

Many of the constructions of the properties defined in Section
\ref{section-e-structures} require us to use \(\sup\) on the \(k\)-image
of some function, and its \(k\)-truncated inclusion map. For this, the
\(k\)-image needs to be \(U\)-small. We therefore need to assume that
\(V\) is appropriately locally \(U\)-small in several of the following
theorems.

\hypertarget{empty-set}{%
\subsection{Empty set}\label{empty-set}}

The empty set is a special case of unordered tupling but it does not
require any assumption about local smallness. Therefore we note this
special case separately.

\begin{theorem}[\agdalink{https://elisabeth.stenholm.one/univalent-material-set-theory/fixed-point.empty-set.html\#1153}]

\label{empty-set-fixed-point}The ∈-structure \((V, ∈)\) has
\textbf{empty set}.\end{theorem}

\begin{proof}

The empty type \(\emptytype\) embeds into any type, in particular, it
embeds into \(V\). Let \(f : \emptytype → V\) be the unique map from
\(\emptytype\) to \(V\). Since the map is an embedding, it is also an
\((n-1)\)-truncated map. The empty set is defined as \begin{align*}
      ∅ ≔ \sup␣(\emptytype, f)
 \end{align*} For \(x : V\) we have the following chain of equivalences:
\begin{equation*}
      x ∈ ∅ ≃ \fib␣f␣x ≃ \left(∑_{s : \emptytype} f␣s = x\right) ≃ \emptytype \qedhere
 \end{equation*}\end{proof}

\hypertarget{restricted-separation-1}{%
\subsection{Restricted separation}\label{restricted-separation-1}}

\begin{theorem}[\agdalink{https://elisabeth.stenholm.one/univalent-material-set-theory/fixed-point.restricted-separation.html\#3727}]

\label{separation-fixed-point}The ∈-structure \((V, ∈)\) has
\textbf{\(U\)-restricted \(n\)-separation}.\end{theorem}

\begin{proof}

For \(x : V\) and \(Φ : \El x → \nType{(n-1)}_U\) we construct the term
\(\{x \mid  Φ\}\) as follows: \begin{equation*}
   \{x \mid Φ\} ≔ \sup␣\left(∑_{a : \overline{x}} Φ␣(\widetilde{x}␣a,(a,\refl)),
   λ(a,\_).\widetilde{x}␣a \right).
 \end{equation*} The map \(λ(a,\_).\widetilde{x}␣a\) is \((n-1)\)
truncated as it is the composition of the \((n-1)\)-truncated map
\(\widetilde{x}\) and the projection
\(π₀ : ∑_{a : \overline{x}} Φ␣(\widetilde{x}␣a,(a,\refl)) → \overline{x}\),
which is \((n-1)\)-truncated since \(Φ\) is a family of
\((n-1)\)-truncated types.

For \(z : V\) we then get the following chain of equivalences:

\begin{align*}
   z ∈ \{x \mid Φ\}
      &= ∑_{(a,\_) : ∑_{a : \overline{x}} Φ␣(\widetilde{x}␣a,(a,\refl))} \widetilde{x}␣a = z \\
      &≃ ∑_{(a,\_) : ∑_{a : \overline{x}} \widetilde{x}␣a = z} Φ␣(\widetilde{x}␣a,(a,\refl)) \\
      &≃ ∑_{e : (z ∈ x)} Φ␣(z,e) \qedhere
 \end{align*}\end{proof}

\hypertarget{unordered-tupling}{%
\subsection{Unordered tupling}\label{unordered-tupling}}

For unordered tupling we can construct the \(k\)-truncated version if
\(V\) is \((k+1)\)-locally \(U\)-small, and we can construct the
\(∞\)-truncated version if the indexing type has lower type level than
\(V\).

\begin{theorem}[\agdalink{https://elisabeth.stenholm.one/univalent-material-set-theory/fixed-point.unordered-tupling.html\#3005}]

\label{unordered-tupling-fixed-point}Let \(k : \Nat^∞\) be such that
\(k ≤ n\). If \(V\) is \((k+1)\)-locally \(U\)-small then the
∈-structure \((V , ∈)\) has \textbf{k-unordered I-tupling} for all
\(I : U\).\end{theorem}

\begin{proof}

Let \(I : U\) and \(v : I → V\), and construct the element
\(\{v\}_k : V\) by \[\{v\}_k ≔ \sup␣(\image_{k-1}␣v, \incl_{k-1}␣v)\]
Since \(k ≤ n\), \(\incl_{k-1}␣v\) is an \((n-1)\)-truncated map. Then,
for any \(z : V\), we have the following chain of equivalences:
\begin{align*}
 z ∈ \{v\}_k 
      &= \fib␣(\incl_{k-1}␣v)\ z \\
      &≃ \big\| \fib␣v\ z\, \big\|_{k-1} \\
      &≡ \big\| ∑_{i : I} v␣i = z\, \big\|_{k-1} \qedhere
 \end{align*}\end{proof}

\begin{corollary}[\agdalink{https://elisabeth.stenholm.one/univalent-material-set-theory/fixed-point.unordered-tupling.html\#5236}]

Let \(k : \Nat^∞\) be such that \(k ≤ n\). If \(V\) is \((k+1)\)-locally
\(U\)-small then \(V\) contains the \(k\)-unordered \(1\)-tupling
\(\{x\}_k\), and the \(k\)-unordered \(2\)-tupling \(\{x, y\}_k\), for
any \(x, y : V\).\end{corollary}

\begin{theorem}[\agdalink{https://elisabeth.stenholm.one/univalent-material-set-theory/fixed-point.unordered-tupling.html\#10439}]

\label{ordered-pairs-fixed-point}If \(V\) is \((n+1)\)-locally
\(U\)-small then it has an ordered pairing structure.\end{theorem}

\begin{proof}

Follows from the previous corollary together with Theorem
\ref{empty-set-fixed-point}, Theorem \ref{0-level-ordered-pairs} and
Proposition \ref{V-is-n-type}.\end{proof}

\begin{theorem}[\agdalink{https://elisabeth.stenholm.one/univalent-material-set-theory/fixed-point.unordered-tupling.html\#3196}]

\label{inf-unordered-tupling-fixed-point}Let \(k : \Nat\) be such that
\(k < n\). If \(V\) is \((n+1)\)-locally \(U\)-small, then the
∈-structure \((V , ∈)\) has \textbf{∞-unordered I-tupling} for all
\(I : \nType{k}_U\).\end{theorem}

\begin{proof}

Let \(I : U\) and \(v : I → V\). We construct the element
\(\{v\}_∞ : V\) as in Theorem \ref{unordered-tupling-fixed-point}:
\[\{v\}_∞ ≔ \sup␣(\image_{n-1}␣v, \incl_{n-1}␣v)\] Then, for any
\(z : V\), we have the following chain of equivalences: \begin{align}
 z ∈ \{v\}_∞ 
      &= \fib␣(\incl_{k-1}␣v)\ z \\
      &≃ \big\| \fib␣v\ z\, \big\|_{n-1} \\
      &≡ \big\| ∑_{i : I} v␣i = z\, \big\|_{n-1} \\
      \label{unordered-inf-tupling-1}
      &≃ ∑_{i : I} v␣i = z
 \end{align} where step (\ref{unordered-inf-tupling-1}) is the fact that
\(∑_{i : I} v␣i = z\) is \((n-1)\)-truncated since \(V\) is an
\(n\)-type and \(I\) is a \(k\)-type for some \(k < n\).\end{proof}

\hypertarget{replacement-1}{%
\subsection{Replacement}\label{replacement-1}}

\begin{theorem}[\agdalink{https://elisabeth.stenholm.one/univalent-material-set-theory/fixed-point.replacement.html\#2631}]

\label{replacement-fixed-point}Let \(k : \Nat^∞\) be such that
\(k ≤ n\). If \(V\) is \((k+1)\)-locally \(U\)-small, then the
∈-structure \((V, ∈)\) has \textbf{\(k\)-replacement}.\end{theorem}

\begin{proof}

For \(x:V\) and \(f : \El␣x → V\), let \(φ : \overline{x} → V\) be the
function that sends \(a : \overline{x}\) to \(f␣(a,\refl)\). We define
the following element:
\[\{f(y) \mid y ∈ x\} ≔ \sup␣(\image_{k-1}␣φ, \incl_{k-1}␣φ)\] Since
\(k ≤ n\), \(\incl_{k-1}␣φ\) is an \((n-1)\)-truncated map. For
\(z : V\) we have the following chain of equivalences:

\begin{align}
      z ∈ \{f(y) \mid y ∈ x\} 
      &= \fib␣(\incl_{k-1}␣φ)␣z \\
      &≃ \; \big\| \fib␣φ\ z \big\|_{k-1} \\
      \label{replacement-step}
      &≃ \; \big\| ∑_{s : \El␣x}f␣s=z \big\|_{k-1}
 \end{align}

where (\ref{replacement-step}) uses Proposition
\ref{index-type-lma}.\end{proof}

\hypertarget{union-1}{%
\subsection{Union}\label{union-1}}

\begin{theorem}[\agdalink{https://elisabeth.stenholm.one/univalent-material-set-theory/fixed-point.union.html\#3059}]

\label{union-fixed-point}Let \(k : \Nat^∞\) be such that \(k ≤ n\). If
\(V\) is \((k+1)\)-locally \(U\)-small, then the ∈-structure \((V, ∈)\)
has \textbf{\(k\)-union}.\end{theorem}

\begin{proof}

For \(x : V\), let
\(φ_x : ∑_{a : \overline{x}} \overline{(\widetilde{x}␣a)} → V\) be the
function that sends \((a,b)\) to \(\widetilde{(\widetilde{x}␣a)}␣b\). We
then define the union by
\[⋃_k x ≔ \sup␣(\image_{k-1}␣(φ_x),\incl_{k-1}␣(φ_x))\] Since \(k ≤ n\),
\(\incl_{k-1}␣(φ_x)\) is an \((n-1)\)-truncated map. For \(z : V\) we
thus have the following chain of equivalences:

\begin{align}
   z ∈ ⋃_k x &= \fib␣(\incl_{k-1}␣(φ_x))␣z \\
           &≃ \; \big\| \fib␣φ_x␣z \big\|_{k-1} \\
           &≃ \; \big\| ∑_{a : \overline{x}} \; z ∈ \widetilde{x}␣a \big\|_{k-1} \\
           \label{union-middle-step}
           &≃ \; \big\| ∑_{y : V} \; (z ∈ y) × (y ∈ x) \big\|_{k-1},
 \end{align}

where (\ref{union-middle-step}) uses Proposition
\ref{index-type-lma}.\end{proof}

\hypertarget{exponentiation-1}{%
\subsection{Exponentiation}\label{exponentiation-1}}

The property of having exponentiation is relative to an ordered pairing
structure. It turns out that the construction of ordered pairs does not
matter. \((V,∈)\) has exponentiation for any ordered pairing structure.
In order to prove this though, we first need some lemmas.

\begin{lemma}[\agdalink{https://elisabeth.stenholm.one/univalent-material-set-theory/fixed-point.exponentiation.html\#5691}]

\label{function-eq}Let \(A\) be a type and \(B\) a type family over
\(A\). For any two functions \(φ, ψ : ∏_{a : A} B␣a\) there is an
equivalence \begin{align*}
      \left(φ = ψ\right)
           ≃ ∑_{e : A ≃ A} ∏_{a : A} (a , φ␣a) = (e␣a , ψ␣(e␣a))
 \end{align*} Moreover, this equivalence sends \(\refl\) to
\((\idequiv , \reflhtpy)\).\end{lemma}

\begin{proof}

For \(φ, ψ : ∏_{a : A} B␣a\) we have the following chain of
equivalences: \begin{align}
      \left(φ = ψ\right)
           \label{function-eq-1}
           &≃ \left(φ ∼ ψ\right) \\
           \label{function-eq-2}
           &≃ ∑_{e : A ≃ A} ∑_{r : \id_A ∼ e} ∏_{a : A} \tr{B}{r␣a} (φ␣a) = ψ␣(e␣a) \\
           &≃ ∑_{e : A ≃ A} ∏_{a : A} ∑_{p : a = e␣a} \tr{B}{p} (φ␣a) = ψ␣(e␣a) \\
           \label{function-eq-3}
           &≃ ∑_{e : A ≃ A} ∏_{a : A} (a , φ␣a) = (e␣a , ψ␣(e␣a))
 \end{align} Step (\ref{function-eq-1}) is function extensionality. In
step (\ref{function-eq-2}) we use the fact that the type
\(∑_{e : A ≃ A} \id_A ∼ e\) is contractible, with
\((\idequiv, \reflhtpy)\) as the center of contraction. Finally, step
(\ref{function-eq-3}) is the characterisation of equality in
\(Σ\)-types.

Chasing \(\refl\) along the chain of equivalences, we get:
\begin{align*}
      \refl &↦ \reflhtpy \\
           &↦ (\idequiv, \reflhtpy, \reflhtpy) \\
           &↦ (\idequiv, λ␣a.(\refl, \refl)) \\
           &↦ (\idequiv , \reflhtpy)
 \end{align*}\end{proof}

\begin{lemma}[\agdalink{https://elisabeth.stenholm.one/univalent-material-set-theory/fixed-point.exponentiation.html\#2298}]

\label{graph-trunc-map}Let \(〈-,-〉 : V → V → V\) be an ordered pairing
structure on \(V\). Given a small type \(A : U\) and a type family
\(B : A → \Type\) together with \((n-1)\)-truncated maps
\(f : A ↪_{n-1} V\) and \(g : ∏_{a : A} B␣a ↪_{n-1} V\), there is an
\((n-1)\)-truncated map \begin{align*}
      \gr_{f,g} : \left(∏_{a : A} B␣a\right) ↪_{n-1} V
 \end{align*}\end{lemma}

\begin{proof}

Given \(φ : ∏_{a : A} B␣a\), we first need to construct an element in
\(V\). To this end, we construct the \((n-1)\)-truncated map
\begin{align*}
      &F_φ : A ↪_{n-1} V \\
      &F_φ␣a = 〈 f␣a , g␣a␣(φ␣a) 〉
 \end{align*} This map is \((n-1)\)-truncated as it is the composition
of the maps \(〈-,-〉\), \(λ␣(a,b).(f␣a , g␣a␣b)\) and
\(λ␣a.(a , φ␣a)\). The first is \((-1)\)-truncated by assumption, and
thus \((n-1)\)-truncated. The second is \((n-1)\)-truncated since \(f\)
is \((n-1)\)-truncated and \(g␣a\) is \((n-1)\)-truncated for all
\(a : A\). To see that the last map is \((n-1)\)-truncated, for any pair
\((a , b) : ∑_{a : A} B␣a\) we have the following chain of equivalences:
\begin{align*}
      \left(∑_{a' : A} (a' , φ␣a') = (a , b)\right)
           ≃ \left(∑_{a' : A} ∑_{p : a' = a} \tr{B}{p} (φ␣a') = b\right)
           ≃ \left(φ␣a = b\right)
 \end{align*} where we have used the fact that the type
\(∑_{a' : A} a' = a\) is contractible, with \((a , \refl)\) as the
center of contraction. Since \(V\) is \(n\)-truncated, it follows that
\(∑_{a' : A} (a' , φ␣a') = (a , b)\) is \((n-1)\)-truncated, and thus
\(λ␣a.(a , φ␣a)\) is an \((n-1)\)-truncated map.

This gives us the underlying map \begin{align*}
      &\gr_{f,g} :  \left(∏_{a : A} B␣a\right) → V \\
      &\gr_{f,g}␣φ ≔ \sup␣(A , F_φ)
 \end{align*} What is left is to show that this map is
\((n-1)\)-truncated. For this we show that
\(\ap{\gr_{f,g}} : φ = ψ → \gr_{f,g}␣φ = \gr_{f,g}␣ψ\) is
\((n-2)\)-truncated for all \(φ, ψ : ∏_{a : A} B␣a\).

Let \(φ, ψ : ∏_{a : A} B␣a\), let
\(α : \left(φ = ψ\right) ≃ ∑_{e : A ≃ A}  ∏_{a : A} (a , φ␣a) = (e␣a , ψ␣(e␣a))\)
be the equivalence given by Lemma \ref{function-eq} and let
\(β : \left(∑_{e : A ≃ A} F_φ ∼ F_ψ ∘ e\right) ≃  \left((A , F_φ) = (A , F_ψ)\right)\)
be the equivalence given by the usual characterisation of identity in
Σ-types together with univalence and function extensionality. We start
by observing that \begin{align*}
      \ap{\gr_{f,g}} ∼
      \ap{\sup} ∘\ β\ ∘\ (λ␣(e , H).(e , \ap{λ␣(a , b).〈 f␣a , g␣a␣b 〉} ∘ H)) ∘ α
 \end{align*} All three of \(α\), \(β\) and \(\ap{\sup}\) are
equivalences, and hence \((n-2)\)-truncated maps. For the last map, it
is enough to show that for all \(e : A ≃ A\) and all \(a : A\)
\begin{align*}
      \ap{λ␣(a , b).〈 f␣a , g␣a␣b 〉} : (a , φ␣a) = (e␣a , ψ␣(e␣a)) → F_φ␣a = F_ψ␣(e␣a)
 \end{align*} is \((n-2)\)-truncated. But this follows from the fact
that the map \(λ␣(a , b).〈 f␣a , g␣a␣b 〉\) is \((n-1)\)-truncated, as
it is the composition of two \((n-1)\)-truncated maps.\end{proof}

\begin{theorem}[Exponentiation \agdalink{https://elisabeth.stenholm.one/univalent-material-set-theory/fixed-point.exponentiation.html\#9105}]

\label{exponentiation-fixed-point}The ∈-structure \((V,∈)\) has
\textbf{exponentiation}, for any ordered pairing structure.\end{theorem}

\begin{proof}

Let \(〈-,-〉 : V → V → V\) be an ordered pairing structure on \(V\).
Given \(x,y : V\), we define the element \begin{align*}
      &y^x : V \\
      &y^x ≔ \sup␣(\overline{x} → \overline{y} , \gr_{\widetilde{x},λ␣\_.\widetilde{y}})
 \end{align*} For \(f : V\) we then have the following chain of
equivalences:

\begin{align}
 f ∈ y^x 
      &≃ ∑_{φ : \overline{x} → \overline{y}} f = \gr_{\widetilde{x},λ␣\_.\widetilde{y}}␣φ \\
      \label{fixpoint-exp-step-1}
      &≃ ∑_{φ : \overline{x} → \overline{y}} ∏_{z : V} z ∈ f ≃ z ∈ \gr_{\widetilde{x},λ␣\_.\widetilde{y}}␣φ \\
      &≃ ∑_{φ : \overline{x} → \overline{y}} 
           ∏_{z : V} z ∈ f ≃ 
                ∑_{a : \overline{x}} 〈 \widetilde{x}␣a, \widetilde{y}␣(φ␣a)〉 = z \\
      \label{fixpoint-exp-step-2}
      &≃ ∑_{φ : \El␣x → \El␣y}
           ∏_{z : V} z ∈ f ≃ 
                ∑_{a : \El␣x} 〈 π_0␣a, π_0␣(φ␣a)〉 = z \\
      \label{fixpoint-exp-step-3}
      &≃ \oper␣x␣y␣f
 \end{align}

In step (\ref{fixpoint-exp-step-1}) we use extensionality and in step
(\ref{fixpoint-exp-step-2}) we use Proposition \ref{index-type-lma}.
Finally, in step (\ref{fixpoint-exp-step-3}) we use Proposition
\ref{operation-equiv}.\end{proof}

\hypertarget{natural-numbers}{%
\subsection{Natural numbers}\label{natural-numbers}}

\begin{theorem}[\agdalink{https://elisabeth.stenholm.one/univalent-material-set-theory/fixed-point.natural-numbers.html\#1627}]

\label{fixed-point-nat}The ∈-structure \((V,∈)\) has natural numbers
represented by \(f\), for any \((n - 1)\)-truncated representation
\(f\).\end{theorem}

\begin{proof}

Let \(f : \Nat → V\) be an \((n-1)\)-truncated representation of
\(\Nat\). We construct the following element: \[n ≔ \sup␣(\Nat,f)\] Then
for any \(z : V\) we have
\[z ∈ n ≡ \fib␣\widetilde{n}␣z = \fib␣f␣z \qedhere\]\end{proof}

\hypertarget{the-initial-tux207f_u-algebra}{%
\section{\texorpdfstring{The initial \(\Tⁿ_U\)-algebra
\label{section-initial-algebras}}{The initial \textbackslash Tⁿ\_U-algebra }}\label{the-initial-tux207f_u-algebra}}

The most straightforward way to construct a fixed-point of a functor is
by constructing its initial algebra. For polynomial functors, the
initial algebras are well understood. They are called W-types and are
used to create a variety of inductive data structures, such as models of
set theory in type theory: The initial algebra for \(\T^∞_U\) formed the
basis for Aczel's setoid model of CZF in Martin-Löf type theory
\cite{aczel1978}. The initial algebras for \(\Tⁿ_U\), for finite \(n\),
are subtypes of this type.

However, the functor \(\Tⁿ_U\) is not a polynomial functor, for finite
\(n\). So, its initial algebra will not be a simple W-type. Looking
carefully at the type \(∑_{A:U} A ↪_{n-1} X\) one can see that it is not
strictly positive, since the type
\(A ↪_{n-1} X ≔ ∑_{f : A → X} ∏_{x:X} \isntype{(n-1)} (\fib␣f␣x)\)
contains a negative occurrence of \(X\). Being strictly positive is a
usual requirement for inductive definitions, so it is a bit surprising
that \(\Tⁿ_U\) still has an initial algebra.

In this section, we will construct the initial algebra for \(\Tⁿ_U\),
which then becomes a fixed-point ∈-structure of level \(n\). One
interesting aspect of doing this proof in HoTT is that the carrier for a
\(\Tⁿ_U\)-algebra can lie on any type level.

We start by recalling Aczel's W-type.

\begin{definition}[\agdalink{https://elisabeth.stenholm.one/univalent-material-set-theory/trees.w-types.html\#1620}]

Let \(V^∞ ≔ W_{A : U} A\), and denote its (uncurried) constructor
\(\sup^∞ : \T^∞_U V^∞ → V^∞\).\end{definition}

\textbf{Remark:} The superscript of the name, \(V^∞\), indicates that
there is no bound on the level of the type. In fact, \(V^∞\) has the
same type level as \(U\).

The pair \((V^∞, \sup^∞)\) is the initial algebra for the functor
\(\T^∞_U : \Type → \Type\), and hence a fixed-point of \(\T^∞_U\). Any
fixed-point has a canonical coalgebra structure, and by Lemma
\ref{coalgebra-structures} this gives rise to a \(U\)-like ∈-structure.
Thus, we will denote by \(∈^∞ : V^∞ → V^∞ → \Type\), this elementhood
relation on \(V^∞\). \textbf{Note:} by construction,
\(x ∈^∞ \sup^∞␣(A , f) ≡ \fib␣f␣x\).

This is the ∈-structure explored in \emph{Multisets in type theory}
\cite{gylterud-multisets}, which also introduced some of the axioms of
the previous sections. In \emph{From multisets to sets in Homotopy Type
Theory} \cite{gylterud-iterative} a 0-level ∈-structure was carved out
as a subtype of \(V^∞\). And we will now prove that this 0-level
structure is, as previously suspected, the initial \(\T^0_U\) algebra.
In fact, we will generalise the construction to form a type \(Vⁿ\) for
all \(n : \Nat\) and show that \(Vⁿ\) is the initial algebra for
\(\Tⁿ_U\).

\hypertarget{iterative-n-types}{%
\subsection{\texorpdfstring{Iterative
\(n\)-types}{Iterative n-types}}\label{iterative-n-types}}

The types \(Vⁿ\) form a stratified hierarchy of ∈-structures, where
\(V^∞\) is at the top. This is analogous to how the \(n\)-types in \(U\)
form a hierarchy with \(U\) itself at the top.

\begin{definition}[\agdalink{https://elisabeth.stenholm.one/univalent-material-set-theory/iterative.set.html\#3450}]

Given \(n:\Nat_{-1}\), define by induction on \(V^∞\), a predicate
\(\ittype{n} : V^∞ → \Type\) by:

\begin{center}
      $\ittype{n}\left(\sup^∞␣(A,f)\right) ≔ \left(\isntruncmap{n} f\right)×\left(∏_{a : A}\ittype{n} (f␣a)\right)$
 \end{center}

\end{definition}

This predicate is propositional, and an element of \(V^∞\) for which the
predicate is true is called \emph{an iterative \(n\)-type}.

\textbf{Remark:} We could define iterative \((-2)\)-types analogously,
but there are no elements in \(V^∞\) which satisfy this predicate, as
the existence of such an element would imply that \(V^∞\) is essentially
\(U\)-small.

\begin{definition}[\agdalink{https://elisabeth.stenholm.one/univalent-material-set-theory/iterative.set.html\#4442}]

For \(n : \Nat_{-1}\), let \(V^{n+1} ≔ ∑_{x:V^∞}\ittype{n} x\) denote
the type of iterative \(n\)-types, which is a subtype of
\(V^∞\).\end{definition}

Even though the type \(Vⁿ\) is not itself \(U\)-small, it has
\(U\)-small identity types

\begin{proposition}[\agdalink{https://elisabeth.stenholm.one/univalent-material-set-theory/iterative.set.html\#5655}]

\label{V-n-locally-small}The type \(Vⁿ\) is locally
\(U\)-small.\end{proposition}

\begin{proof}

First, we observe that, since being an iterative \((n-1)\)-type is a
proposition, \(Vⁿ\) is a subtype of \(V^∞\). It was shown in
\emph{Multisets in type theory} \cite[Lemma~3]{gylterud-multisets} that
the latter is locally \(U\)-small. Since the equality of a subtype is
the underlying equality of the base type, it follows that \(Vⁿ\) is
locally \(U\)-small.\end{proof}

In the following proofs we will leave out elements of types which are
propositions, unless they are necessary, in order to increase
readability. For the full details of the proofs, please see the Agda
formalisation.

\begin{proposition}[\agdalink{https://elisabeth.stenholm.one/univalent-material-set-theory/iterative.set.html\#6169}]

The map \(\sup^∞ : \T^∞_U V^∞ → V^∞\) restricts to a map
\(\supⁿ : \Tⁿ_U Vⁿ → Vⁿ\).\end{proposition}

\begin{proof}

Let \(A : U\) and \(f : A ↪_{n-1} Vⁿ\). The element
\(\sup^∞␣(A , π₀ ∘ f)\) is an iterative \((n-1)\)-type:

\begin{itemize}
      \item $π₀ ∘ f$ is $(n-1)$-truncated since $f$ is $(n-1)$-truncated, and
           $π₀$ is the inclusion of a subtype, which is an embedding, and thus
           $(n-1)$-truncated.
      \item For every $a : A$ we have $π₁␣(f␣a) : \ittype{(n-1)} (π₀␣(f␣a))$.
 \end{itemize}

Thus we construct the map \begin{align*}
      &\operatorname{sup}ⁿ : \Tⁿ_U Vⁿ → Vⁿ \\
      &\operatorname{sup}ⁿ␣(A , f) ≔ \left(\operatorname{sup}^∞␣(A , π₀ ∘ f), \longunderscore\right) \qedhere
 \end{align*}\end{proof}

\begin{proposition}[\agdalink{https://elisabeth.stenholm.one/univalent-material-set-theory/iterative.set.html\#6830}]

There is a map \(\desupⁿ : Vⁿ → \Tⁿ_U Vⁿ\).\end{proposition}

\begin{proof}

Let \(x : Vⁿ\), without loss of generality we may assume that
\(x ≡ (\operatorname{sup}^∞␣(A , f) , (p , q))\) for some \(A : U\),
\(f : A → V^∞\), \(p : \isntruncmap{(n-1)} f\) and
\(q : ∏_{a : A} \ittype{(n-1)} (f␣a)\). The following function:
\begin{align*}
      λ␣a.(f␣a, q␣a) : A → Vⁿ
 \end{align*} is \((n-1)\)-truncated. This is because the composition
\(π₀ ∘ (λ␣a.(f␣a, q␣a))\) is \((n-1)\)-truncated, by \(p\), and \(π₀\)
is \((n-1)\)-truncated as it is the inclusion of a subtype. This implies
that the right factor, \(λ␣a.(f␣a, q␣a)\), is \((n-1)\)-truncated. Thus
we construct the map \begin{align*}
      &\desupⁿ : Vⁿ → \Tⁿ_U Vⁿ \\
      &\desupⁿ␣(\operatorname{sup}^∞␣(A , f) , (p , q))
           ≔ (A , λ␣a.(f␣a, q␣a)) \qedhere
 \end{align*}\end{proof}

\begin{proposition}[\agdalink{https://elisabeth.stenholm.one/univalent-material-set-theory/iterative.set.html\#8048}]

\label{sup-desup-equiv}The maps \(\supⁿ\) and \(\desupⁿ\) form an
equivalence.\end{proposition}

\begin{proof}

For \((\operatorname{sup}^∞␣(A , f) , (p , q)) : Vⁿ\) we have
\begin{align*}
      π₀␣(\operatorname{sup}ⁿ␣&(\desupⁿ␣(\operatorname{sup}^∞␣(A , f) , (p , q)))) \\
           &≡ π₀␣(\operatorname{sup}ⁿ␣(A , λ␣a.(f␣a, q␣a))) \\
           &≡ (\operatorname{sup}^∞␣(A , f)) \\
           &≡ π₀␣(\operatorname{sup}^∞␣(A , f) , (p , q))
 \end{align*} Thus, by the characterisation of equality in subtypes,
\(\supⁿ ∘ \desupⁿ ∼ \id\).

For \(A : U\) and \(f : A ↪_{n-1} Vⁿ\) we have \begin{align*}
      π₀␣(\desupⁿ␣(\operatorname{sup}ⁿ␣(A , f))) ≡ A
 \end{align*} and \begin{align*}
      π₀␣(π₁&(\desupⁿ␣(\operatorname{sup}ⁿ␣(A , f)))) \\
           &≡ π₀␣(π₁(\desupⁿ␣(\operatorname{sup}^∞␣(A , π₀ ∘ f), (\_,π₁ ∘ f)))) \\
           &≡ λ␣a.(π₀␣(f␣a),π₁␣(f␣a)) \\
           &≡ π₀␣f
 \end{align*} Thus by the characterisation of equality in \(Σ\)-types,
and equality in subtypes, \(\desupⁿ ∘ \supⁿ ∼ \id\).\end{proof}

\begin{theorem}[\agdalink{https://elisabeth.stenholm.one/univalent-material-set-theory/iterative.set.html\#8458}]

\label{V-n-fixed-point}\(Vⁿ\) is a fixed-point to the functor
\(\Tⁿ_U\).\end{theorem}

\begin{proof}

Corollary of the previous proposition.\end{proof}

\begin{theorem}

\label{thm-n-lvl-e-str-Vn}There is an ∈-structure \((Vⁿ, ∈ⁿ)\), on
\(Vⁿ\), of level \(n\).\end{theorem}

\begin{proof}

This follows from Theorem \ref{V-n-fixed-point} together with Theorem
\ref{U-like-equiv-T-n-coalgebra}.\end{proof}

\textbf{Remark:} By construction, for \(x, y : Vⁿ\) we have
\(y ∈ⁿ x ≡ π₀␣y ∈^∞ π₀␣x\).

\textbf{Remark:} The ∈-structure \((V⁰,∈⁰)\) is exactly the ∈-structure
studied in \emph{From multisets to sets in Homotopy Type Theory}
\cite{gylterud-iterative}, which there was established to be equivalent
to the model of set theory given in the HoTT Book \cite{hottbook}.

\hypertarget{induction-principle-and-recursors-for-vux207f}{%
\subsection{\texorpdfstring{Induction principle and recursors for
\(Vⁿ\)}{Induction principle and recursors for Vⁿ}}\label{induction-principle-and-recursors-for-vux207f}}

Vⁿ is a composite type: a Σ-type over a W-type and a predicate involving
induction on that W-type. Since it is not a primitive inductive type it
does not come with a ready-to-use induction principle. But, to some
extent, \(\supⁿ : \Tⁿ_U␣Vⁿ → Vⁿ\) acts as a constructor, for which we
have a corresponding induction principle. The expected computation rules
do not hold definitionally, but are instead proven to hold up to
identity.

\begin{proposition}[Elimination for Vⁿ \agdalink{https://elisabeth.stenholm.one/univalent-material-set-theory/iterative.set.html\#9361}]

\label{V-n-elim}Given any family \(P : Vⁿ → \Type\) and a function

\begin{center}
      $φ : ∏_{A : U} ∏_{f : A ↪_{n-1} Vⁿ} \left(∏_{a:A} P␣(f␣a)\right) → P␣(\supⁿ␣(A,f))$
 \end{center}

there is a function \(\elim_{Vⁿ}␣P\ φ : ∏_{x:Vⁿ} P␣x\).

Furthermore, there is a path
\(\elim_{Vⁿ}␣P\ φ␣(\supⁿ (A , f)) =  φ␣A␣f␣(\elim_{Vⁿ}␣P␣φ ∘ f)\), for
any \(A : U\) and \(f : A ↪_{n-1} Vⁿ\).\end{proposition}

\begin{proof}

Given \(x : Vⁿ\), we may assume that
\(x ≡ (\operatorname{sup}^∞␣(A,f) , (p , q))\) where \(A : U\),
\(f : A → V^∞\), \(p : \isntruncmap{(n-1)} f\) and
\(q : ∏_{a : A}  \ittype{(n-1)} (f␣a)\). The function
\(\elim_{Vⁿ}␣P\ φ\) is defined as follows \begin{align*}
      \elim_{Vⁿ}␣P\ φ\ x ≔ \tr{P}{α}{\left(φ␣A␣(π₁␣(\desupⁿ␣x))␣(λ␣a.\elim_{Vⁿ}␣P\ φ\ (f␣a , q␣a))\right)}
 \end{align*} where \(α : \supⁿ␣(\desupⁿ␣x) = x\) by Proposition
\ref{sup-desup-equiv}. To construct the path, we use univalence to do
equivalence induction. Let \(Q\) be the type family \begin{align*}
      &Q : ∏_{B : \Type} Vⁿ ≃ B → \Type \\
      &Q␣B␣e ≔ ∏_{p : ∏_{b : B} P␣\left(e^{-1}␣b\right)} 
                ∏_{b : B} \tr{P}{α}{\left(p␣\left(e␣\left(e^{-1}␣b\right)\right)\right)} = p␣b
 \end{align*} where
\(α : e^{-1}␣\left(e␣\left(e^{-1}␣b\right)\right) = e^{-1}␣b\) is the
proof that \(e\) is a retraction applied to the element \(e^{-1}␣b\).
Then, for the case \(B ≔ Vⁿ\) and \(e ≔ \idequiv\) we have
\begin{align*}
      Q␣Vⁿ␣\idequiv ≡ ∏_{p : ∏_{x : Vⁿ} P␣(a)} ∏_{x : Vⁿ} (p␣x = p␣x)
 \end{align*} which is inhabited by \(λ␣p.\reflhtpy\). By univalence we
thus have an element of the type \(Q␣B␣e\) for any type \(B\) and
equivalence \(e : Vⁿ ≃ B\). For \(B ≔ Vⁿ\) and \(e ≔ \desupⁿ\) we
construct the term \begin{align*}
      &p : ∏_{(A' , f') : \Tⁿ_U Vⁿ} P␣(\operatorname{sup}ⁿ␣(A' , f')) \\
      &p␣(A' , f') ≔ φ␣A'␣f'␣(\elim_{Vⁿ}␣P\ φ\ ∘ f')
 \end{align*} Then we have an element of the type
\(Q␣\left(\Tⁿ_U Vⁿ\right)␣\desupⁿ␣p\ (A , f)\), which unfolds to
\begin{align*}
      \left(\tr{P}{α}{\left(φ␣A␣(π₁␣(\desupⁿ␣x))␣(λ␣a.\elim_{Vⁿ}␣P\ φ\ (f␣a , q␣a))\right)}
      = φ␣A␣f␣(\elim_{Vⁿ}␣P\ φ\ ∘ f)\right)
 \end{align*} where
\(α : \supⁿ␣\left(\desupⁿ␣\left(\supⁿ␣(A' , f')\right)\right) = \supⁿ␣(A' , f')\)
is the same element as was used in the construction of
\(\elim_{Vⁿ}␣P\ φ\).\end{proof}

As usual, we can consider the specialisation from eliminators to
recursors, and in the recursors the computation rules end up holding
definitionally. Since \((n-1)\)-truncated maps are also functions,
\(Vⁿ\) actually has recursors for both \(\T^∞_U\)-algebras and
\(\Tⁿ_U\)-algebras.

\begin{proposition}[Untruncated recursion for Vⁿ \agdalink{https://elisabeth.stenholm.one/univalent-material-set-theory/iterative.set.html\#10751}]

Given any type \(X\) and map \(m : \T^∞_U X → X\) there is a map
\(\trec_{Vⁿ}␣X\ m :  Vⁿ → X\) such that for \((A , f) : \Tⁿ_U Vⁿ\) the
following definitional equality holds: \begin{align*}
      \trec_{Vⁿ}␣X\ m\ (\operatorname{sup}ⁿ␣(A , f))
           ≡ m␣(\T^∞_U (\trec_{Vⁿ}␣X\ m)\ (A , f))
 \end{align*}\end{proposition}

\begin{proof}

Given \(x : Vⁿ\), we may assume that
\(x ≡ (\operatorname{sup}^∞␣(A,f) , (p , q))\) where \(A : U\),
\(f : A → V^∞\), \(p : \isntruncmap{(n-1)} f\) and
\(q : ∏_{a : A}  \ittype{(n-1)} (f␣a)\). The map \(\trec_{Vⁿ}\) is
defined as: \begin{align*}
      \trec_{Vⁿ}␣X\ m\ x ≔ m␣(A , λ␣a.\trec_{Vⁿ}␣X\ m\ (f␣a , q␣a))
 \end{align*} For \((A , f) : \Tⁿ_U Vⁿ\) we thus have the following
chain of definitional equalities: \begin{align*}
      \trec_{Vⁿ}␣X\ m\ (\operatorname{sup}ⁿ␣(A , f))
           &≡ \trec_{Vⁿ}␣X\ m\ (\operatorname{sup}^∞␣(A, π₀ ∘ f) , (\_, π₁ ∘ f)) \\
           &≡ m␣(A , \trec_{Vⁿ}␣X\ m ∘ f) \\
           &≡ m␣(\T^∞_U (\trec_{Vⁿ}␣X\ m)\ (A , f)) \qedhere
 \end{align*}\end{proof}

\begin{corollary}[Truncated recursion for Vⁿ \agdalink{https://elisabeth.stenholm.one/univalent-material-set-theory/iterative.set.html\#11356}]

Given a \(\Tⁿ_U\)-algebra \((X , m)\) there is a map
\(\tnrec_{Vⁿ}␣(X , m) : Vⁿ → X\) such that for \((A , f) : \Tⁿ_U Vⁿ\)
the following definitional equality holds: \begin{align*}
      \tnrec_{Vⁿ}␣(X , m)␣(\operatorname{sup}ⁿ␣(A , f)) 
           ≡ m␣(\Tⁿ_U (\tnrec_{Vⁿ}␣(X , m))\ (A , f))
 \end{align*}\end{corollary}

\begin{proof}

The map \(\tnrec_{Vⁿ}␣(X , m)\) is defined as: \begin{align*}
      \tnrec_{Vⁿ}␣(X , m) 
           ≔ \trec_{Vⁿ}␣X\ (λ␣(A , f).m␣(\image_{n-1} f , \incl_{n-1} f))
 \end{align*} For \((A , f) : \Tⁿ_U Vⁿ\) we have the following chain of
definitional equalities: \begin{align*}
      \tnrec_{Vⁿ}␣(X , m)␣(\operatorname{sup}ⁿ␣(A , f))
           &≡ \trec_{Vⁿ}␣X\ m'␣(\operatorname{sup}ⁿ␣(A , f)) \\
           &≡ m'␣(A , \trec_{Vⁿ}␣X\ m' ∘ f) \\
           &≡ m␣(\image_{n-1} (\trec_{Vⁿ}␣X\ m' ∘ f) , \incl_{n-1} (\trec_{Vⁿ}␣X\ m' ∘ f)) \\
           &≡ m␣(\image_{n-1} (\tnrec_{Vⁿ}␣(X , m) ∘ f) , \incl_{n-1} (\tnrec_{Vⁿ}␣(X , m) ∘ f))\\
           &≡ m␣(\Tⁿ_U (\tnrec_{Vⁿ}␣(X , m))\ (A , f))
 \end{align*} where
\(m' ≔ λ␣(A , f).m␣(\image_{n-1} f , \incl_{n-1} f)\).\end{proof}

\hypertarget{initiality-of-vux207f}{%
\subsection{Initiality of Vⁿ}\label{initiality-of-vux207f}}

In the proof of initiality we are to show that the type of homomorphisms
from the initial algebra, into any other algebra, is contractible. If we
were working with mere sets, it would be sufficient to show that the
underlying maps of the homomorphisms are equal to a specified canonical
map. But being a homomorphism is actually a structure when the types
involved are of higher levels. So, the proof of contractibility has to
coherently transfer this structure when proving that every homomorphism
is equal to the center of contraction. This is achieved by using the
induction principle for \(Vⁿ\) and the characterisation of the identity
type on \(\Tⁿ_U\)-algebra homomorphisms.

\begin{theorem}[\agdalink{https://elisabeth.stenholm.one/univalent-material-set-theory/iterative.set.html\#13409}]

\label{V-n-initial}The algebra \((Vⁿ,\supⁿ)\) is initial: given any
other \(\Tⁿ_U\)-algebra \((X,m)\) the type of algebra homomorphisms from
\((Vⁿ,\supⁿ)\) to \((X,m)\) is contractible.\end{theorem}

\begin{proof}

The center of contraction is given by
\(\left(\tnrec_{Vⁿ}␣(X ,  m),\reflhtpy\right)\). We will use the
characterisation of equality between \(\Tⁿ_U\)-algebra homomorphisms
given by Lemma \ref{lma-T-n-alg-hom-eq} to show that any other
homomorphism is equal to \(\tnrec_{Vⁿ}␣(X , m)\).

Let \((φ , α)\) be another \(\Tⁿ_U\)-algebra homomorphism from
\((Vⁿ,\supⁿ)\) to \((X,m)\). We need to do two things:

\begin{itemize}
      \item Construct a homotopy $ε : φ ∼ \tnrec_{Vⁿ}␣(X , m)$.
      \item For each $(A,f) : \Tⁿ_U␣Vⁿ$, construct a path
           \begin{center}
                $α␣(A , f) · \ap{m} \left(\Tⁿ_U ε␣(A , f)\right)
                = ε \left(\operatorname{sup}ⁿ␣(A , f)\right) · \reflhtpy␣(A , f).$
           \end{center}
 \end{itemize}

To construct a homotopy from \(φ\) to \(\tnrec_{Vⁿ}␣(X , m)\) we use the
elimination principle on \(Vⁿ\), by Proposition \ref{V-n-elim}, with the
type family \(P␣x ≔ \left(φ␣x = \tnrec_{Vⁿ}␣(X , m)␣x\right)\). This
means that given \(A : U\), \(f : A ↪_{n-1} Vⁿ\) and
\(H : ∏_{a : A} φ␣(f␣a) =  \tnrec_{Vⁿ}␣(X , m)␣(f␣a)\), we need to
construct a path \begin{align*}
      φ␣(\operatorname{sup}ⁿ␣(A , f)) = \tnrec_{Vⁿ}␣(X , m)␣(\operatorname{sup}ⁿ␣(A , f))
 \end{align*}

We have the following chain of equalities: \begin{align}
      φ␣(\operatorname{sup}ⁿ␣(A , f))
           \label{T-rec-unique-1}
           &= m␣(\Tⁿ_U φ␣(A , f)) \\
           &≡ m␣(\image_{n-1} (φ ∘ f) , \incl_{n-1} (φ ∘ f)) \\
           \label{T-rec-unique-2}
           &= m␣(\image_{n-1} (\tnrec_{Vⁿ}␣(X , m) ∘ f) , \incl_{n-1} (\tnrec_{Vⁿ}␣(X , m) ∘ f)) \\
           &≡ m␣(\Tⁿ_U (\tnrec_{Vⁿ}␣(X , m))␣(A , f)) \\
           &≡ \tnrec_{Vⁿ}␣(X , m)␣(\operatorname{sup}ⁿ␣(A , f))
 \end{align} Step (\ref{T-rec-unique-1}) is the path \(α␣(A , f)\) and
step (\ref{T-rec-unique-2}) is the path
\(\ap{λ␣h.m␣(\image_{n-1} h , \incl_{n-1} h)} (\funext␣H)\). So let
\begin{align*}
      σ ≔ λ␣A␣f␣H.α␣(A , f) · \ap{λ␣h.m␣(\image_{n-1} h , \incl_{n-1} h)} (\funext␣H)
 \end{align*} then by Proposition \ref{V-n-elim} we have a homotopy
\(\elim_{Vⁿ}␣P\ σ : φ ∼ \tnrec_{Vⁿ}␣(X , m)\).

It remains to construct the second component of the \(Σ\)-type in Lemma
\ref{lma-T-n-alg-hom-eq}. We have the following chain of equalities:
\begin{align}
      α␣(A , f) &· \ap{m} \left(\Tⁿ_U \left(\elim_{Vⁿ}␣P\ σ\right)␣(A , f)\right) \\
           &≡ α␣(A , f) · \ap{m} \left(\ap{λ␣ψ.\Tⁿ_U ψ (A , f)} \left(\funext␣\left(\elim_{Vⁿ}␣P\ σ\right)\right)\right) \\
           \label{T-rec-unique-3}
           &= α␣(A , f) · \ap{m} \left(\ap{λ␣h.(\image_{n-1} h , \incl_{n-1} h)}
                \left(\ap{-∘ f} \left(\funext␣\left(\elim_{Vⁿ}␣P\ σ\right)\right)\right)\right) \\
           \label{T-rec-unique-4}
           &= α␣(A , f) · \ap{λ␣h.m␣(\image_{n-1} h , \incl_{n-1} h)}
                \left(\ap{-∘ f} \left(\funext␣\left(\elim_{Vⁿ}␣P\ σ\right)\right)\right) \\
           \label{T-rec-unique-5}
           &= α␣(A , f) · \ap{λ␣h.m␣(\image_{n-1} h , \incl_{n-1} h)}
                \left(\funext␣\left(\left(\elim_{Vⁿ}␣P\ σ\right) ∘ f\right)\right) \\
           &≡ σ␣A␣f␣\left(\left(\elim_{Vⁿ}␣P\ σ\right) ∘ f\right) \\
           \label{T-rec-unique-6}
           &= \elim_{Vⁿ}␣P\ σ \left(\operatorname{sup}ⁿ␣(A , f)\right) \\
           &= \elim_{Vⁿ}␣P\ σ \left(\operatorname{sup}ⁿ␣(A , f)\right) · \reflhtpy␣(A , f)
 \end{align} In steps (\ref{T-rec-unique-3}) and (\ref{T-rec-unique-4})
we use the fact that \(\operatorname{ap}\) and function composition
commute. Step (\ref{T-rec-unique-5}) uses the fact that function
extensionality respects precomposition. Finally, step
(\ref{T-rec-unique-6}) is the computation rule for
\(\elim_{Vⁿ}\).\end{proof}

\hypertarget{properties}{%
\subsection{Properties}\label{properties}}

Since \(Vⁿ\) is a fixed-point for \(\Tⁿ_U\), the induced ∈-structure
satisfies the properties as shown in Section
\ref{section-fixed-point-models}. But as \(Vⁿ\) is the initial algebra
for \(\Tⁿ_U\), it also satisfies foundation, which is not true of all
fixed-points for \(\Tⁿ_U\). More explicitly, we list the properties
which \((Vⁿ, ∈ⁿ)\) satisfies.

\begin{theorem}[\agdalink{https://elisabeth.stenholm.one/univalent-material-set-theory/iterative.set.properties.html}]

For \(n : \Nat^∞\), the ∈-structure \((Vⁿ, ∈ⁿ)\) satisfies the following
properties:

\begin{itemize}
      \item empty set,
      \item $U$-restricted $n$-separation,
      \item ∞-unordered $I$-tupling, for all $k : \Nat_{-1}$ such that $k < n$ and
            $k$-truncated types $I : U$,
      \item $k$-unordered $I$-tupling, for all $k : \Nat_{-1}$ such that $k ≤ n$ and $I : U$,
      \item $k$-replacement, for all $k : \Nat_{-1}$ such that $k ≤ n$,
      \item $k$-union, for all $k : \Nat_{-1}$ such that $k ≤ n$,
      \item exponentiation, for any ordered pairing structure,
      \item natural numbers for any $(n-1)$-truncated representation,
      \item foundation.
 \end{itemize}

\end{theorem}

\begin{proof}

The first two properties follow directly from Theorem
\ref{empty-set-fixed-point} and Theorem \ref{separation-fixed-point}
respectively since \(Vⁿ\) is a fixed-point for \(\Tⁿ_U\). The next four
properties follow from Theorem \ref{inf-unordered-tupling-fixed-point},
Theorem \ref{unordered-tupling-fixed-point}, Theorem
\ref{replacement-fixed-point} and Theorem \ref{union-fixed-point}
respectively, together with the fact that \(Vⁿ\) is locally small
(Proposition \ref{V-n-locally-small}), and thus \((k+1)\)-small for all
\(k : \Nat\).

By Theorem \ref{exponentiation-fixed-point}, \((Vⁿ, ∈ⁿ)\) has
exponentiation for any ordered pairing structure. (Note that
\((Vⁿ, ∈ⁿ)\) has at least one ordered pairing structure by Theorem
\ref{ordered-pairs-fixed-point}.)

For natural numbers, the result follows from Theorem
\ref{fixed-point-nat}, since \(Vⁿ\) is a fixed-point for \(\Tⁿ_U\).

Lastly, we use the induction principle for \(Vⁿ\) to show that
\((Vⁿ, ∈ⁿ)\) has foundation. Given \(A : U\) and \(f : A ↪_{n-1} Vⁿ\) we
need to construct an element of type
\[ \left(∏_{a:A} \Acc␣(f␣a)\right) → \Acc␣(\operatorname{sup}ⁿ␣(A,f)). \]
Therefore suppose we have \(p : ∏_{a:A} \Acc␣(f␣a)\). We construct the
following element:
\[\acc␣\left(λ␣y␣(a , q).\tr{Acc}{q}{(p␣a)}\right) : \Acc␣(\operatorname{sup}ⁿ␣(A,f)).\]
It then follows by Proposition \ref{V-n-elim} that we have
\[∏_{x:Vⁿ} \Acc␣x. \qedhere\]\end{proof}

\hypertarget{vux207f-as-an-n-type-universe-of-n-types}{%
\section{\texorpdfstring{\(\Vⁿ\) as an \(n\)-type universe of
\(n\)-types
\label{section-tarski-universe}}{\textbackslash Vⁿ as an n-type universe of n-types }}\label{vux207f-as-an-n-type-universe-of-n-types}}

We have seen that the type \(\Vⁿ\) can be equipped with a binary
relation ∈ⁿ, making it a model of our higher level generalisation of
material set theory. There is a second perspective on \(\Vⁿ\), namely as
a type theoretic universe à la Tarski. This has already been explored in
detail for the type \(\V⁰\) in a previous paper
\cite{gratzer2024category}, showing that it is a mere set universe of
mere sets which is closed under all the usual type formers and which has
definitional decoding. Here we will show that the corresponding universe
construction can be done for every \(\Vⁿ\).

Let \((\Vⁿ,∈ⁿ)\) be the \(n\)-level ∈-structure given by Theorem
\ref{thm-n-lvl-e-str-Vn}. Note that the type \(x ∈ⁿ y\) is an
\((n-1)\)-type, for any elements \(x, y : \Vⁿ\).

\begin{proposition}[\agdalink{https://elisabeth.stenholm.one/univalent-material-set-theory/fixed-point.core.html\#1223}]

\label{Vn-is-n-type}The type \(\Vⁿ\) is an \(n\)-type.\end{proposition}

\begin{proof}

This follows both from Proposition \ref{m-h-level}, and from Proposition
\ref{V-is-n-type} together with Theorem
\ref{V-n-fixed-point}.\end{proof}

\begin{definition}[\agdalink{https://elisabeth.stenholm.one/univalent-material-set-theory/iterative.set.html\#7329}]

For \(n : ℕ^∞\), define the decoding function on \(\Vⁿ\) as the family
\begin{align*}
    &\Elⁿ : \Vⁿ → U \\
    &\Elⁿ␣x ≔ \overline{x}
\end{align*}\end{definition}

We take this as the decoding function rather than \(\El\) from
Definition \ref{def-el}. This is so that the decoding holds up to
definitional equality, since we have
\[\Elⁿ␣(\operatorname{sup}ⁿ␣(A,f)) ≡ A.\] But the two families are
equivalent by Proposition \ref{index-type-lma}.

\begin{proposition}[\agdalink{https://elisabeth.stenholm.one/univalent-material-set-theory/iterative.set.html\#17929}]

\label{Vn-decodes-into-n-types}For \(n : ℕ^∞\), the decoding \(\Elⁿ␣a\)
of any element \(a : \Vⁿ\), is an \(n\)-type.\end{proposition}

\begin{proof}

The map \(\widetilde{a}\) is an \((n-1)\)-truncated map from \(\Elⁿ␣a\)
into \(\Vⁿ\). The domain \(\Elⁿ␣a\) is equivalent to the total space of
fibers of \(\widetilde{a}\), which is an \(n\)-type as the base,
\(\Vⁿ\), is an \(n\)-type and each fiber is an
\((n-1)\)-type.\end{proof}

Propositions \ref{Vn-is-n-type} and \ref{Vn-decodes-into-n-types} thus
show that \(\Vⁿ\) is an \(n\)-type universe of \(n\)-types. Section
\ref{representations-of-types-in-e-structures} explored internalisations
of types in ∈-structures. The type of internalisations of a type was
shown to be equivalent to the type of internalisable representations of
that type (Proposition \ref{is-prop-int-of-repr}). In the case of
\(\Vⁿ\), the internalisable representations are the \((n-1)\)-truncated
ones.

\begin{proposition}[\agdalink{https://elisabeth.stenholm.one/univalent-material-set-theory/iterative.set.html\#18692}]

For \(n : ℕ^∞\) and a type \(A : U\) together with a representation
\(f : A → \Vⁿ\), the representation is internalisable if and only if it
is \((n-1)\)-truncated.\end{proposition}

\begin{proof}

Suppose that \(f\) is internalisable, i.e.~we have an element of the
type \[∑_{a:\Vⁿ} ∏_{z:\Vⁿ}z∈ⁿa ≃ \fib f␣z.\] Since the type \(z ∈ⁿ a\)
is an \((n-1)\)-type for all \(z : \Vⁿ\) it follows that \(\fib f␣z\) is
an \((n-1)\)-type for all \(z : \Vⁿ\). Conversely, suppose that \(f\) is
\((n-1)\)-truncated. Then we take for \(a\) the element \(\supⁿ␣(A,f)\)
and for the family of equivalences, the family of identity
equivalences.\end{proof}

\begin{proposition}[\agdalink{https://elisabeth.stenholm.one/univalent-material-set-theory/iterative.set.html\#19119}]

\label{internalisations-Vn}For \(n : ℕ^∞\) and \(A : U\), the type of
\((n-1)\)-truncated representations of \(A\), \(A ↪_{n-1} \Vⁿ\), is
equivalent to the type of internalisations of \(A\),
\(∑_{a : A} \Elⁿ␣a ≃ A\).\end{proposition}

\begin{proof}

This follows by Proposition \ref{equiv-int-repr} and the fact that the
families \(\El\) and \(\Elⁿ\) are equivalent, together with the previous
proposition. Alternatively, \(\supⁿ\) is an equivalence making the
following diagram commute

\begin{center}
    \begin{tikzcd}
        \Poⁿ_U␣\Vⁿ \arrow[rr, "\supⁿ"] \arrow[rd, "π₀"'] &   & \Vⁿ \arrow[ld, "\Elⁿ"] \\
                                                         & U &                       
    \end{tikzcd}
\end{center}

Therefore the type of fibers of \(π₀\) over \(A\), which is
\(A ↪_{n-1} \Vⁿ\), is equivalent to the type of fibers of \(\Elⁿ\) over
\(A\), which is equivalent to \(∑_{a : A} \Elⁿ␣a ≃ A\) by
univalence.\end{proof}

The universes are cumulative both with regards to universe levels and
with regards to type levels. So far we have assumed only two universes,
\(U\) and \(\Type\), but assume for the next proposition a (cumulative)
hierarchy of universes \(U₀, U₁, ⋯, U_{\ell}, ⋯\) and let \(\Vⁿ_{\ell}\)
denote the initial algebra to the functor \(\Poⁿ_{U_{\ell}}\).

\begin{proposition}[\agdalink{https://elisabeth.stenholm.one/univalent-material-set-theory/iterative.set.html\#16978}]

For \(n : ℕ^∞\), there is an internalisation of \(\Vⁿ_{\ell}\) in
\(\Vⁿ_{\ell + 1}\).\end{proposition}

\begin{proof}

The map \begin{align*}
    &φ : \V^∞_{\ell} → \V^∞_{\ell + 1} \\
    &φ␣(\operatorname{sup}^∞␣(A,f)) ≔ \operatorname{sup}^∞␣(A,φ∘f)
\end{align*} is an embedding, as shown in a related paper
\cite{gratzer2024category}. To show that it restricts to iterative
\(n\)-types, let \(\sup^∞␣((A,f),\longunderscore) : \Vⁿ_{\ell}\),
i.e.~\(f\) is \((n-1)\)-truncated and \(f␣a\) is an iterative \(n\)-type
for all \(a : A\). Then \(φ ∘ f\) is \((n-1)\)-truncated as it is the
composition of two \((n-1)\)-truncated maps. Moreover, by the induction
hypothesis \(φ␣(f␣a)\) is an iterative \(n\)-type since \(f␣a\) is an
iterative \(n\)-type, for all \(a : A\).

The map φ is therefore an embedding \(\Vⁿ_{\ell} ↪ \Vⁿ_{\ell + 1}\). By
Proposition \ref{internalisations-Vn} the map gives rise to an
internalisation of \(\Vⁿ_{\ell}\), namely
\(\sup^{n+1}␣(\Vⁿ_{\ell},φ)\).\end{proof}

\begin{proposition}[\agdalink{https://elisabeth.stenholm.one/univalent-material-set-theory/iterative.set.html\#5084}]

For \(n : ℕ^∞\), there is an embedding
\(\Vⁿ ↪ \V^{n+1}\).\end{proposition}

\begin{proof}

This is simply the fact that an iterative \(n\)-type is also an
iterative \((n+1)\)-type.\end{proof}

Note that there is a size issue with internalising \(\Vⁿ\) in
\(\V^{n+1}\). If there was an element \(v : \V^{n+1}\) such that
\(\El^{n+1}␣v ≃ \Vⁿ\), then the type \(\Vⁿ\) would be essentially
\(U\)-small, which would induce a paradox. However, assuming again a
hierarchy of universes, the type \(\Vⁿ_{\ell}\) can be internalised in
\(\V^{n+1}_{\ell + 1}\).

The universe \(\Vⁿ\) also contains all the usual types and type formers,
assuming they exist in the underlying universe \(U\).

\begin{proposition}[\agdalink{https://elisabeth.stenholm.one/univalent-material-set-theory/iterative.set.html\#23577}]

For \(n : ℕ^∞\), the universe \(\Vⁿ\) contains the following types and
type formers:

\begin{itemize}
\tightlist
\item
  the empty type, unit type and booleans,
\item
  the natural numbers,
\item
  Π-types,
\item
  Σ-types,
\item
  coproducts and
\item
  identity types.
\end{itemize}

\end{proposition}

\begin{proof}

For the empty type, unit type and booleans, we internalise them as the
elements ∅, \(\{∅\}₀\) and \(\{∅,\{∅\}₀\}₀\) respectively, using
Theorems \ref{empty-set-fixed-point} and
\ref{unordered-tupling-fixed-point}. The natural numbers have an
internalisation as they have an \((n-1)\)-truncated representation by
Example \ref{von-neumann-nats}.

For Π-types and Σ-types, suppose we have an element \(a : \Vⁿ\) and a
map \(b : \Elⁿ␣a → \Vⁿ\). In the first case, note that
\(\gr_{\,\widetilde{a},λ␣i.\widetilde{(b␣i)}}\) in Lemma
\ref{graph-trunc-map} gives an \((n-1)\)-truncated representation of
\(∏_{i : \Elⁿ␣a} \Elⁿ␣(b␣i)\). For Σ-types, we use the ordered pairing
structure given by Theorem \ref{ordered-pairs-fixed-point}. The map
\(λ␣(i,j).〈\widetilde{a}␣i,\widetilde{(b␣i)}␣j〉\) is an
\((n-1)\)-truncated representation of \(∑_{i : \Elⁿ␣a} \Elⁿ␣(b␣i)\).

Given \(a, b : \Vⁿ\), the map \begin{align*}
    &f : \Elⁿ␣a + \Elⁿ␣b → \Vⁿ \\
    &f␣(\inl␣i) ≔ 〈∅,\widetilde{a}␣i〉 \\
    &f␣(\inr␣j) ≔ 〈\{∅\}₀,\widetilde{b}␣j〉
\end{align*} is \((n-1)\)-truncated. To see this, note that for
\(z : \Vⁿ\) we have the following chain of equivalences: \begin{align}
    \fib␣f␣z &≡ \left(∑_{s : \Elⁿ␣a + \Elⁿ␣b} f␣s = z\right) \\
        &≃\left(∑_{i : \Elⁿ␣a} 〈∅,\widetilde{a}␣i〉 = z\right)
            + \left(∑_{j : \Elⁿ␣b} 〈\{∅\}₀,\widetilde{b}␣j〉 = z\right) \\
        &≃\left(∑_{((s , t) , p) : \fib␣〈-,-〉␣z} \left(\fib␣\widetilde{a}␣t\right) × \left(s = ∅\right)\right) \\
            &\hspace{10pt}+ \left(∑_{((s , t) , p) : \fib␣〈-,-〉␣z} \left(\fib␣\widetilde{b}␣t\right) × \left(s = \{∅\}₀\right)\right) \nonumber
\end{align} The last type is \((n-1)\)-truncated as all the components
are \((n-1)\)-truncated and the two summands are mutually exclusive
since \(∅ ≠ \{∅\}₀\). The map \(f\) therefore gives an internalisation
of the coproduct \(\Elⁿ␣a + \Elⁿ␣b\). For identity types we note that
the map \(λ␣p.␣∅ : a = b → \Vⁿ\) is \((n-1)\)-truncated as it is a
function from an \((n-1)\)-type into an \(n\)-type. So the type
\(a = b\) is internalisable.\end{proof}

Note that all the codes in \(\Vⁿ\) are constructed using \(\supⁿ\).
Therefore, the decoding of a code is definitionally equal to the type
being encoded.

As \(\Vⁿ\) is an \(n\)-type universe of \(n\)-types, a natural question
to ask is which \(n\)-types lie in the universe. For \(\V¹\) we have
seen that the classifying space of any group can be internalised. In
particular, we have an internalisation of the higher inductive type of
the circle in \(\V¹\). Using the same idea, we can internalise any
\(n\)-type in \(\V^{n+1}\).

\begin{proposition}[\agdalink{https://elisabeth.stenholm.one/univalent-material-set-theory/iterative.set.html\#20626}]

For \(n : ℕ^∞\), any \(U\)-small \(n\)-type can be internalised in
\(\V^{n+1}\).\end{proposition}

\begin{proof}

Let \(A : U\) be an \(n\)-type. Let \(f : A → \V^{n+1}\) be the constant
function sending any element to ∅. This is an \(n\)-truncated map as the
domain is an \(n\)-type and the codomain is an \((n+1)\)-type. The
element \(\sup^{n+1}␣(A,f)\) is thus an internalisation of \(A\) in
\(\V^{n+1}\).\end{proof}

For \(\V⁰\), the statement that any mere set has an internalisation is
essentially the axiom of wellfounded materialisation
\cite{shulman_stack_2010}. The statement that any \(n\)-type can be
internalised in \(\Vⁿ\) can thus be thought of as a higher level
generalisation of wellfounded materialisation. It amounts to being able
to equip every \(n\)-type with some higher iterative structure.

\begin{proposition}

For \(n : ℕ^∞\), the universe \(\Vⁿ\) is \textbf{not} a univalent
universe.\end{proposition}

\begin{proof}

The quickest way to see this is to note that \(\Vⁿ\) contains (at least)
two distinct internalisations of the unit type: \(\{∅\}₀\) and
\(\{\{∅\}₀\}₀\) (but there are of course many more). As these both
decode to the unit type but are not equal, univalence fails.\end{proof}

From the reasoning above it follows that \(\Vⁿ\) does not even have
partial univalence \citep{sattler_partial_2020} (univalence restricted
to \(k\)-types, for some \(k\)). The reason why univalence fails is
essentially that the decoding of an element in \(\Vⁿ\) only returns the
indexing type of the children to the root, seeing the element as a tree,
regardless of the rest of the tree. Thus, several trees can internalise
the same type, but be distinct as trees, i.e.~as elements of the
universe.

\hypertarget{conclusion-and-future-work}{%
\section{\texorpdfstring{Conclusion and future
work\label{section-conclusion}}{Conclusion and future work}}\label{conclusion-and-future-work}}

In this paper we defined the concept of an ∈-structure and gave
generalisations of the axioms of constructive set theory to higher level
structures. As instances of such higher level ∈-structures, we
generalised the construction of the type of iterative sets to obtain the
initial algebras, \(\Vⁿ\), to the functors \(\Poⁿ_U\). These were shown
to model the ∈-structure properties at the same level or lower.
Moreover, \(\Vⁿ\) was shown to be an \(n\)-type universe of \(n\)-types
with definitional decoding.

\hypertarget{related-work}{%
\subsection{Related work}\label{related-work}}

Gallozzi \cite{gallozzi} constructs a family of interpretations of set
theory in homotopy type theory indexed on two type levels, \(k\) and
\(h\). His interpretation is a generalisation of Aczel's model. In one
direction, he generalises by taking as the type of sets Aczel's W-type,
but over the small universe of \(k\)-types, rather than the whole
universe \(U\). In the other direction, he uses \(h\)-truncated Σ-types,
while Aczel uses untruncated Σ-types. Gallozzi then shows that this
models Aczel's CZF \cite{aczel1978} and Myhill's CST
\cite{myhill_constructive_1975}.

These models are setoid models\hspace{1pt}---\hspace{1pt}equality is not
interpreted as the identity type, as opposed to our models. They are not
∈-structures in our sense either, in that our version of extensionality
does not hold. Of course, he shows that his interpretation of
extensionality holds. The \(k\)-level W-type used by Gallozzi is a
\((k+1)\)-type, and it is a subtype of our type \(\V^{k+1}\), so any set
in that universe is also a set in our universe.

The HoTT Book model of set theory is equivalent to the iterative sets
model. It is not clear, however, how to generalise that construction as
we have done with the iterative sets, as this would require more complex
higher inductive types.

\hypertarget{future-work}{%
\subsection{Future work}\label{future-work}}

As we have seen, the universe \(\Vⁿ\) embeds into the universe
\(\V^{n+1}\). So the higher universes contain the types of the lower
universes. However, it remains to see which new types appear at each
level. For instance, we showed that the higher inductive type of the
circle, and more generally the classifying space of any group, lies in
\(\V¹\). For the universe \(\Vⁿ\) one would like to construct an
internalisation of (at least some) proper \(n\)-types. This amounts to
giving the type some kind of higher iterative structure, which would be
an interesting direction for further study.

\printbibliography
\end{document}